\documentclass[a4paper,11pt,english]{smfart}
\usepackage{amsfonts}
\usepackage{amsmath,amsthm} 
\usepackage{hyperref} 
\usepackage{latexsym}
\usepackage{array}
\usepackage{mathrsfs}
\usepackage{amssymb}
\usepackage{enumerate}
\usepackage[francais,english]{babel}
\usepackage{color}

\theoremstyle{plain}  \newtheorem{theorem}{Theorem}[section]
\newtheorem{lemma}[theorem]{Lemma}
\newtheorem{proposition}[theorem]{Proposition}
 
\newtheorem{definition}[theorem]{Definition} \theoremstyle{remark}
\newtheorem{remark}[theorem]{Remark}

\newcommand{\lsim}{  \lesssim   }

\newcommand{\de}{  \delta   }

\renewcommand{\(}{   \left[\hspace{-1ex}\left[\hspace{0.5ex}  }
\renewcommand{\)}{  \hspace{0.5ex} \right]\hspace{-1ex}\right] }
\newcommand{\R}{  \mathbb{R}   }

\newcommand{\eps}{\varepsilon}

\newcommand{\C}{  \mathbb{C}   }
\newcommand{\Z}{  \mathbb{Z}   }
\newcommand{\N}{  \mathbb{N}   }

\newcommand{\A}{  \mathcal{A}   }
\newcommand{\M}{  \mathcal{M}   }

\newcommand{\Ca}{  \mathcal{C}   }

\newcommand{\E}{  \mathcal{E}   }
\newcommand{\NF}{  \mathcal{NF}   }

\renewcommand{\O}{  \mathcal{O}   }

\newcommand{\T}{  \mathbb{T}   }
\newcommand{\Tc}{  \mathcal{T}   }
\newcommand{\D}{  \mathcal{D}   }

\newcommand{\Cc}{  \mathcal{C}   }
\newcommand{\dd}{  \text{d}   }

\newcommand{\om}{  \omega   }

\renewcommand{\a}{  \alpha   }
\newcommand{\p}{  \partial   }
\renewcommand{\b}{  \beta   }
\renewcommand{\P}{  \mathcal{P}   }
\newcommand{\ga}{\gamma   }
\newcommand{\s}{  \sigma   }
\newcommand{\lan}{  \langle  }
\newcommand{\ran}{  \rangle  }

\newcommand{\ka}{  \kappa   }
\renewcommand{\r}{  \rho   }

\newcommand{\la}{  \lambda_a   }
\newcommand{\lb}{  \lambda_b   }

\newcommand{\zz}{  \mathfrak z  }
\renewcommand{\phi}{  \varphi  }
\renewcommand{\L}{  \mathcal{L}   }
\newcommand{\Lc}{  \hat{\mathcal{L}  } }
\renewcommand{\S}{  \mathbb{S}   }
\newcommand{\nazz}{ (\nabla_\rho\cdot {\mathfrak z} ) }
\newcommand{\diag}{\operatorname{diag}}
\newcommand{\meas}{\operatorname{meas}}
\newcommand{\card}{\operatorname{card}}

\newcommand{\Span}{\operatorname{Span}}
\newcommand{\be}{\begin{equation}}
\newcommand{\ee}{\end{equation}}
\newcommand{\ben}{\begin{equation*}}
\newcommand{\een}{\end{equation*}}
\newcommand{\ban}{\begin{align*}}
\newcommand{\ean}{\end{align*}}
\numberwithin{equation}{section}

\def\norma#1{\left\| #1\right\|}

\newcommand{\msb}{ \mathcal M_{s,\beta} }
\newcommand{\ff}{ [f]^{s}_{\s,\mu,\D} }
\newcommand{\fb}{ [f]^{s,\b}_{\s,\mu,\D} }
\newcommand{\fbp}{ [f]^{s,\b+}_{\s,\mu,\D} }
\newcommand{\Tbp}{\mathcal{T}^{s,\b+}(\s,\mu,\D)}
\newcommand{\Tb}{\mathcal{T}^{s,\b}(\s,\mu,\D)}

\newcommand{\hh}{ [h]^{s}_{\s,\mu,\D} }

\newcommand{\hhh}{ [h]^{s,\b}_{\s,\mu,\D} }

 \author{ Beno\^it Gr\'ebert}
\address{Laboratoire de Math\'ematiques Jean Leray, Universit\'e de Nantes, UMR CNRS 6629\\
2, rue de la Houssini\`ere \\
44322 Nantes Cedex 03, France}
\email{benoit.grebert@univ-nantes.fr}
\author{Eric Paturel}
\address{Laboratoire de Math\'ematiques Jean Leray, Universit\'e de Nantes, UMR CNRS 6629\\
2, rue de la Houssini\`ere \\
44322 Nantes Cedex 03, France}
\email{eric.paturel@univ-nantes.fr}

\title{KAM for the Klein Gordon equation on $\S^d$.}
 
\begin{document}

\begin{abstract}
Recently the KAM theory has been extended to multidimensional PDEs. Nevertheless all these recent results   concern PDEs on the torus, essentially because in that case the corresponding linear PDE is diagonalized in the Fourier basis and the structure of the resonant sets is quite simple. In the present paper,  we consider an important physical example that do not fit in this context:  the Klein Gordon equation on $\S^d$. Our abstract KAM theorem also allow to prove the reducibility of the corresponding linear operator with time quasiperiodic potentials.
  \end{abstract}
  

\subjclass{ }
\thanks{
}

\maketitle
\tableofcontents
\section{Introduction.}

If the KAM theorem is now well documented for nonlinear Hamiltonian PDEs in 1-dimensional context (see \cite{kuk87,Kuk2,pos89}) only few results exist for multidimensional PDEs.

Existence of quasi-periodic solutions of space-multidimensional  PDE were first proved in \cite{B1} (see also 
  \cite{B2}) but with a technique based on the Nash-Moser thorem that does not allow to analyze the linear stability of the obtained solutions.
 Some KAM-theorems for small-amplitude solutions of multidimensional beam equations (see \eqref{beam}   above)
  with typical $m$   were obtained in
  \cite{GY05, GY06}. Both works treat equations  with a constant-coefficient nonlinearity 
  $g(x,u)=g(u)$, which  is significantly easier than the general case. 
  The first complete KAM theorem for space-multidimensional PDE was obtained in \cite{EK10}. Also see \cite{Ber1, Ber2}. \\ 
  The techniques  developed by Eliasson-Kuksin have been improved in \cite{EGK1,EGK2} to allow a KAM result without external parameters.  In these two papers the authors prove the existence of small amplitude quasi-periodic solutions of the beam equation on the $d$-dimensional torus. They further investigate  the stability of these solutions and give explicit examples where the solution is linearly unstable and thus exhibits hyperbolic features (a sort of whiskered torus).\\
All these examples concern PDEs on the torus, essentially because in that case the corresponding linear PDE is diagonalized in the Fourier basis and the structure of the resonant sets is  the same for NLS, NLW or beam equation. In the present paper, adapting the technics in \cite{EK10}, we consider an important example that do not fit in the Fourier analysis: the Klein Gordon equation on the sphere $\mathbb S^d$ .\\
Notice that  existence of quasi-periodic solutions for NLW and NLS on compact Lie groups via  Nash Moser technics (and  without linear stability) has been proved recently in \cite{BP,BCP}.

\medskip

To understand the new difficulties, let us start with a brief overview of the method developed in \cite{EK10}.  Consider the nonlinear Schr\"odinger equation on $\T^d$
$$iu_t=-\Delta u + \text{nonlinear terms}, \quad x\in \T^d,\ t\in\R.$$
In Fourier variables it reads\footnote{The space $\Z^d$ is equipped with standard euclidian norm: $|k|^2=k_1^2+\cdots k_d^2$.}
$$i \, \dot u_k= |k|^2 u_k +\text{nonlinear terms},\quad k\in\Z^d.$$
So two Fourier modes indexed by $k,j\in\Z^d$ are (linearly) resonant when $|k|^2=|j|^2$. For the beam equation on the torus, the resonance relation is the same. The resonant sets $\E_k=\{j\in\Z^d\mid |j|^2=|k|^2\}$ define a  natural clustering of $\Z^d$. All the modes in the block $\E_k$ have the same energy,  and we can expect that the interactions between different blocks are small, but the interactions inside a block could be of order one.
With this idea in mind, the principal step of the KAM technique, i.e. the resolution of the so called homological equation, leads to the inversion of an infinite matrix which is block-diagonal with respect to this clustering. 
 It turns out that these blocks have cardinality growing with $|k|$ making harder the control of the inverse of this matrix. As a consequence we lose regularity each time we solve the homological equation. Of course, this is not acceptable for an infinite induction. The very nice idea in \cite{EK10} consists in considering a sub-clustering constructed as the equivalence classes of the equivalence relation on $\Z^d$ generated by the pre-equivalence relation
$$ a\sim b \Longleftrightarrow \left\{\begin{array}{l} |a|=|b| \\   {|a-b|}
 \leq \Delta \end{array}\right.$$
Let $[a]_\Delta$ denote the equivalence class of $a$. 
The crucial fact (proved in \cite{EK10}) is that the blocks are {\it finite}
with a maximal ``diameter''
$$\max_{[a]_\Delta=[b]_\Delta} |a-b|\leq C_d \Delta^{\frac{(d+1)!}2} $$
depending only on $\Delta$. With such a clustering, we do not lose regularity when we solve the homological equation. 
Furthermore, working in a phase space of analytic functions $u$ or equivalently, exponentially decreasing Fourier coefficients $u_k$, it turns out that the homological equation is "almost" block diagonal relatively to this clustering. 
Then we let the parameter $\Delta$ grow at each step of the KAM iteration. 

\medskip

Unfortunately, this estimate of the diameter of a block $[a]_\Delta$ by a constant independent of $|a|$ is a sort of miracle that does not persist in other cases. For instance if we consider the nonlinear Klein Gordon equation on the sphere  $\S^2$,
$$(\partial_{t}^2-\Delta+m)u=\text{nonlinear terms},\quad t\in\R,\ x\in\S^2$$
then the linear part diagonalizes in the harmonic basis $\Psi_{j,\ell}$ (see Section 3) and the natural clustering is given by
the resonant sets $\{(j,\ell)\in\N^2\mid \ell=-j,\cdots,j \}$. We can easily convince ourself that there is no simple construction of a sub-clustering compatible with the equation, in such a way that the size of the blocks does no more depend on the energy.\\
So we have to invent a new way to proceed. First we  consider a phase space $Y_s$ with polynomial decay on the Fourier coefficient (corresponding to Sobolev regularity for $u$) instead of exponential decay and we use a different norm on the Hessian matrix that takes into account the polynomial decrease of the off-diagonal blocks:
\be\label{star}
|M|_{\b,s}=\sup_{j,k\in\N}{\|M_{[k]}^{[j]}\|}(kj)^\beta \left(\frac{\min(j,k)+|j^2-k^2|}{\min(j,k)}\right)^{s/2}\ee
where $[j]=\{(n,m)\in\N^2\mid n+m= j \}$ is the block of energy $j$, $M_{[k]}^{[j]}$ is the interaction matrix $M$ reduced to the eigenspace of energy $j$ and of energy $k$, and $\|\cdot\|$ is the operator norm in $\ell^2$. This norm was suggested by our study of the Birkhoff normal form in \cite{BDGS} and \cite{GIP}.\\
 This technical changes make disappear the loss of regularity in the resolution of the homological equation. Nevertheless this is not the end of the story, since this Sobolev structure of the phase space $\Tc^{s,\b}$ (see Section \ref{2}) is not stable by Poisson bracket and thus is not adapted to an iterative scheme. So the second ingredient consists in a trick previously used in \cite{GT}: we take advantage of the regularizing effect of the homological equation to obtain a solution in a slightly more regular space $\Tc^{s,\b+}$ and then we verify that $\{\Tc^{s,\b},\Tc^{s,\b+}\}\in \Tc^{s,\b}$ (see Section 4) which enables an iterative procedure. The last problem is to check that the non linear term, say $P$, belongs to the class $\Tc^{s,\b}$ which imposes a decreasing condition on the operator norm of the blocks of the Hessian of $P$. It turns out that this condition is satisfied for the Klein Gordon equation  on spheres (and also on Zoll manifold, see Remark \ref{rem-Zoll}). A similar condition  is also satisfied for the quantum harmonic oscillator  on $\R^d$
 $$i\, u_t=-\Delta u +|x|^2u+\text{nonlinear terms}, \quad x\in\R^d.$$
 But unfortunately, in order to belong in the class $\Tc^{s,\b}$, the gradient of the nonlinear term has to be regularizing, a fact that is not true for the quantum harmonic oscillator, and thus our KAM theorem does not apply in this case. Nevertheless, this last condition is not required when $P$ is quadratic and thus this method allows to obtain a reducibility result for the quantum harmonic oscillator with time quasi periodic potential. This is detailed in our forthcoming paper \cite{GP2}. \\
 In this paper we only consider PDEs with external parameters (similar to a convolution potential in the case of NLS on the torus). Following \cite{EGK2} we could expect to remove these external parameters (and to use only internal parameters) but the technical cost would be very high.

\medskip

We now state our result for the Klein Gordon equation on the sphere.
Denote by $\Delta$ the Laplace-Beltrami
operator on the sphere $\mathbb S^d$, $m>0$ and let $\Lambda_0=(-\Delta +m)^{1/2}$.   
The spectrum of $\Lambda_0$ equals
$\{\sqrt{j(j+d-1)+m}\mid\ j\geq 0\}.$
For each $j\geq 1$ let $E_j$ be the associated eigenspace, its dimension is $ d_j=O(j^{d-1})$. We denote by $\Psi_{j,l}$  the  harmonic function of degree $j$ and order $\ell$ so that we have
$$E_j=\Span\{\Psi_{j,l},\  l=1,\cdots,d_j\}.$$
We denote
$$\E:=\{(j,\ell)\in\N\times\Z\mid j\geq 0\text{ and }\ell=1,\cdots,d_j\}$$
in such a way that $\{\Psi_{a},\, a\in\E\}$ is a basis of $L^2(\S^d, \C)$.\\
We introduce the harmonic multiplier $M_\r$
defined on the basis $(\Psi_a)_{a\in\E}$ of $L^2(\S^d)$ by
\be\label{MKG}
M_\r \Psi_a=\r_a\Psi_a\quad \text{ for } a\in\E
\ee
where $(\r_a)_{a\in\E}$ is a bounded sequence of  nonnegative real numbers.

Let $g$ be a real analytic function on $\mathbb S^d\times \R$  such that $g$ vanishes at
least at order 2 in the second variable at the origin. We consider the following 
 nonlinear Klein Gordon equation 
\be
\label{KG} 
(\partial_{t}^2-\Delta+m+\de M_\r)u+\eps g(x,u)=0,\quad t\in\R,\ x\in\S^d 
\ee
where $\de>0$ and $\eps>0$ are   small parameters.\\
Introducing $\Lambda=(-\Delta +m+\de M_\r)^{1/2}$ and $v=-u_t\equiv - \dot u$, \eqref{KG} reads
\ben \left\{\begin{array}{ll}
 \dot u &= - 
 v,\\
 \dot v &=\Lambda^2 u    +\eps g(x,u).
\end{array}\right. 
\een
Defining 
 $
 \psi =\frac 1{\sqrt 2}(\Lambda^{1/2}u  + i\Lambda^{-1/2}v) $
 we get
$$
\frac 1 i \dot \psi =\Lambda \psi+ \frac{\eps}{\sqrt 2}\Lambda^{-1/2}g\left(x,\Lambda^{-1/2}\left(\frac{\psi+
\bar\psi}{\sqrt 2}\right)\right)\,.
$$
Thus, if we endow the space   $L^2(\S^d, \C)$ with the standard  real symplectic structure given by the two-form
$\ 
-id\psi\wedge d\bar \psi $
 then equation 
 \eqref{KG} becomes a Hamiltonian system 
$$\dot \psi=i\frac{\partial H}{\partial \bar\psi}$$
with the hamiltonian function
$$
H(\psi,\bar\psi)=\int_{\S^d}(\Lambda \psi)\bar\psi \dd x +\eps\int_{\S^d}G\left(x,\Lambda^{-1/2}\left(\frac{\psi+\bar\psi}{\sqrt 2}\right)\right)\dd x.
$$
where $G$ is a primitive of $g$ with respect to the variable $u$: $g=\partial_u G$.\\
The linear operator $\Lambda$ is diagonal in the  basis  $\{\Psi_{a}, \, a\in\E\}$:
$$
\Lambda \Psi_{a}=\la \Psi_a,\;\;\la= \sqrt{w_a(w_a+d-1)+m+\de\r_a},
\qquad  \forall\,a\in\E
$$
 where we set
$$w_{(j,\ell)}=j\quad \forall\,(j,\ell)\in\E.$$
Let us decompose $\psi$ and $\bar\psi$  in the   basis $\{\Psi_a, \, a\in\E\}$:
$$
\psi=\sum_{a\in\E}\xi_a \Psi_a,\quad \bar\psi=\sum_{a\in\E}\eta_a \Psi_{a}\,.
$$
On $\P_\C:=\ell^2(\E,\C)\times\ell^2(\E,\C)$ endowed with the  complex 
symplectic structure
${ -}i\sum_s \dd\xi_s\wedge\dd\eta_s$ we consider the Hamiltonian system
\be \label{KG2} \left\{\begin{array}{ll}\dot \xi_a&=i\frac{\partial H}{\partial \eta_a}\\ \dot \eta_a&=-i\frac{\partial H}{\partial \xi_a}\end{array}\right. \quad a\in\E\ee
where the Hamiltonian function $H$ is given by
\be\label{HKG}H=\sum_{a\in\E}\la \xi_a\eta_a+\eps\int_{\S^d}G\left(x,\sum_{a\in\E}\frac{(\xi_a+\eta_a)\Psi_a}{\sqrt{ 2}\ \la^{1/2}}\right)\dd x.\ee
The Klein Gordon equation \eqref{KG} is then  equivalent to the  Hamiltonian system \eqref{KG2} restricted to the real subspace 
$$\P_\R:=\{(\xi,\eta)\in \ell^2(\E,\C)\times\ell^2(\E,\C)\mid \eta_a=\bar\xi_a, \ a\in\E\}.$$
\begin{definition}
Let $\A\subset \E$ a finite subset of cardinal $n$. This set is {\em admissible} if and only if
\be\label{admissible}\A\ni(j_1,\ell_1)\neq(j_2,\ell_2)\in\A \Rightarrow j_1\neq j_2.\ee
\end{definition}
We fix $I_a\in[1,2]$ for $a\in\A$, the initial $n$ actions, and we write the modes $\A$ in action-angle variables:
$$\xi_a=\sqrt{I_a+r_a}e^{i\theta_a},\quad \eta_a=\sqrt{I_a+r_a}e^{-i\theta_a}.$$
We define $\L=\E\setminus\A$ 
and, to simplify the presentation, we assume  that
$$\r_{j,l}=\r_j \text{ for } (j,\ell)\in \A\ ;\ \r_{j,l}=0 \text{ for } (j,\ell)\in \L.$$
Set
\begin{align}\begin{split}\label{nota}
w_{j,\ell}&=j\quad \text{ for }(j,\ell)\in\E,\\
\lambda_{j,\ell}&= \sqrt{j(j+d-1)+m}\text{ for }(j,\ell)\in\L,\\
(\om_0)_{j,\ell}(\r)&=\sqrt{j(j+d-1)+m+\de\r_j}\text{ for }(j,\ell)\in\A,\\
\zeta&=(\xi_a,\eta_a)_{a\in\L}.
\end{split}\end{align}
With this notation $H$ reads (up to a constant)
$$H(r,\theta,\zeta)= \lan\om_0(\r),r\ran+\sum_{a\in\L} \la\xi_a\eta_a+ \eps f(r,\theta,\zeta)$$
where 
$$f(r,\theta,\zeta)=
\int_{\S^d}G\left(x, \hat u(r,\theta,\zeta)(x)\right)dx$$
and 
\be\label{uhat}
\hat u(r,\theta,\zeta)(x)=\sum_{a\in\A}\frac{\sqrt{2(I_a+r_a)}\cos{\theta_a}}{\la^{1/2}}\Psi_a(x)+\sum_{a\in\L}\frac{(\xi_a+\eta_a)}{\sqrt{ 2}\ \la^{1/2}}  \Psi_a(x) .\ee

 Let us set $ u_1(\theta,x) = \hat u(0,\theta;0)(x) $.
 Then  for any $I\in[1,2]^n$ and $\theta_0\in\T^n$ 
 the function $(t,x)\mapsto u_1(\theta_0+t\om,x)$ is a quasi-periodic solution of \eqref{KG} with
 $\eps=0$.  Our main theorem 
 states that for most external
 parameter $\r$ this quasi-periodic solution persists (but is  slightly deformed) when we turn on the nonlinearity:
\begin{theorem}\label{thmKG} Fix $n$ the cardinality of an admissible set $\A$, $s>1$ the Sobolev regularity and $g$ the nonlinearity. There exists an exponent $\upsilon(d) >0$ such that, for $\eps$ sufficiently small (depending on $n$, $s$ and $g$) and  satisfying
$$\eps\leq \de^{\upsilon(d)}\,,$$
 there exists a Borel subset $ \D'$, positive constants $\a$ and $C$ with
$$\D'\subset [1,2]^n, \quad
\meas([1,2]^n\setminus\D')\leq C\eps^{\a},$$ 
such that for  $\rho\in\D'$, there is a function 
$ u(\theta,x)$, analytic in $\theta\in\T^n_{\frac\s 2}$ and smooth in $x\in\S^d$, satisfying
$$\sup_{|\Im\theta|<\frac\s 2}\|u(\theta,\cdot)-u_1(\theta,\cdot)\|_{H^{s}(\S^d)}
\leq \eps^{11/12},$$
and there is a mapping  $$\om':\D'\to \R^n,\quad
 \|\om'-\om\|_{C^1(\D')}\leq \eps,$$
such that for any $\r\in \D'$ the function 
$$
u(t,x)=u(\theta+t\om'(\r),x)
$$
is a   solution of the Klein Gordon equation \eqref{KG}. Furthermore this solution  is linearly stable.\\ 
 The positive 
  constant $\a$  depends only on  $n$ while $C$ also depends on $g$ and $s$.
 \end{theorem}
Notice that in this work we did not try to optimize the exponents. In particular $11/12$ could be replaced by any number strictly less than 1 and the choice of $\upsilon(d)$ obtained by inserting \eqref{d0} in \eqref{ed} is far from optimal. Actually we could expect that $\eps\ll \delta$ is sufficient but the technical cost would be very high. This effort is justified when we try to prove a KAM result without external parameters (see \cite{KP} where the authors obtained a condition of the form $\eps\ll \delta$ in the context of the NLS equation; see also   \cite{EGK1}, \cite{EGK2} for the beam equation and  \cite{Bou} for the 1d wave equation where the authors obtained a condition of the form $\eps\ll \delta^{1+\a}$ for suitable $\a>0$ ).

We will deduce Theorem \ref{thmKG} from an abstract KAM result stated in Section 2 and proved in Section 6. The application to the Klein Gordon equation is detailed in Section 3. Roughly speaking, our abstract theorem  applies to any multidimensional PDE with regularizing nonlinearity and which satisfies the second Melnikov condition (see Hypothesis A3). For instance, it doesn't apply to nonlinear Schr\"odinger on any compact manifold since we have no regularizing effect in that case.  On the contrary, it applies to the beam equation on the torus $\T^d$ (see Remark \ref{rem-beam}). Unfortunately it doesn't apply to the nonlinear wave equation on $\T^d$ (see Remark \ref{rem-NLW}), since in that case the second Melnikov condition is not satisfied.\\
In Section 4 we study the Hamiltonian flows generated by Hamiltonian functions in $\Tc^{s,\beta}$. In  Section 5 we detail the resolution of the homological equation. In both Sections 4 and 5 we use techniques and proofs that were developed in \cite{EK10} and \cite{EGK1}. The novelty lies in the use  of  different norms (see \eqref{star}) and the use of two different classes of Hamiltonians: $\Tc^{s,\beta}$ and $\Tc^{s,\beta+}$ which, of course, complicate the technical arguments.  For convenience of the reader we repeat most of the proofs.  We point out that, for the resolution of the homological equation (Section 5), we use a variant of a Lemma due to Delort-Szeftel \cite{DS1}, whose proof is given in Appendix A.

\medskip

\noindent {\it Acknowledgement:} The authors acknowledge the support from the project ANAE (ANR-13-BS01-0010-03)
of the Agence Nationale de la Recherche.

\section{Setting and abstract KAM theorem.}\label{2}
\noindent {\bf Notations.} In this section we state a KAM result for a Hamiltonian $H=h+ f$ of the following form
$$
H= \lan\om(\r), r \ran+\frac 1 2 \langle \zeta,  A(\r)\zeta\rangle +  f(r,\theta,\zeta;\r)$$
where 
\begin{itemize}
\item $\om \in \R^n$ is the frequencies vector corresponding to the internal modes in action-angle variables $(r,\theta)\in\R^n_+\times \T^n$.
\item  $\zeta=(\zeta_s)_{s\in\L}$ are the external modes:    $\L$ is an infinite set of indices, $\zeta_s=(p_s,q_s)\in\R^2$ and $\R^2$ is endowed with the standard symplectic structure $dq\wedge dp$.
\item $A$ is a linear operator acting on the external modes, typically $A$ is diagonal.
\item  $f$ is a perturbative Hamiltonian depending on all the modes and is of order $\eps$ where
 $\eps$ is a small parameter.
\item  $\r$ is an external parameter in $\D$ a compact subset of $\R^p$ with $p\geq n$.
\end{itemize}
We now detail the structures behind these objects and the hypothesis needed for the KAM result.\\

\noindent {\bf Cluster structure on $\L$.}
Let $\L$ be a set of indices and   $w:\L\to \N \setminus\{ 0\}$ be an "energy" function\footnote{We could replace the assumption that $w$ takes integer values by $\{w_a-w_b\mid a,b\in\L\}$ accumulates on a discrete set. }
 on $\L$.
We consider the clustering of $\L$ given by 
$\L=\cup_{a\in\L}[a]$ associated to equivalence relation 
$$b\sim a \Longleftrightarrow w_a=w_b.$$
We denote  $\hat\L=\L/\sim$.
We  assume that the cardinal of each energy level is finite and  that there exist $C_{b}>0$ and  $d^{*}>0$ two constants such that the cardinality of $[a]$ is controlled by $C_{b}w_a^{d}$:
\be\label{block}d_a=d_{[a]}=\card\{b\in\L\mid w_b=w_a\}\leq C_{b}w_a^{d^{*}}.\ee

 \smallskip

\noindent {\bf Linear space.}
Let  $s\geq 0$, we consider the complex  weighted $\ell_2$-space
$$
Y_s=\{\zeta=(\zeta_a\in\C^2,\ a\in \L)\mid \|\zeta\|_s<\infty\} 
$$
where\footnote{We provide $\C^2$ with the hermitian norm, $|\zeta_a|=|(p_a,q_a)|=\sqrt{|p_a|^2+|q_a|^2}$.}
$$
\|\zeta\|_s^2=\sum_{a\in\L}|\zeta_a|^2 w_a^{2s}.$$

In the spaces $Y_s$ acts the linear operator $J$,
$$
J\ :\ \{\zeta_a\}\mapsto \{\s_2\zeta_a\}, \quad \text{with } \s_2=\left(\begin{array}{cc} 0&-1\\1&0\end{array}\right).$$
 It provides the spaces $Y_s$, $s\geq 0$, with the symplectic structure $J\dd\zeta\wedge\dd \zeta$. To any $C^1$-smooth function defined on a domain $\O\subset Y_s$,  corresponds the Hamiltonian equation 
 $$
 \dot \zeta =J\nabla f(\zeta),$$
 where $\nabla f$ is the gradient with respect to the scalar product in $Y$.
 
 \smallskip
 
  \noindent {\bf Infinite matrices.}
 We denote by $\M_{s,\b}$ the set of infinite matrices $A:\L\times \L\to \M_{2\times 2}(\R)$ with value in the space of real $2\times 2$ matrices that are symmetric
$$
A_a^{b}={}^t\hspace{-0,1cm}A_{b}^a,\quad \forall a,\ b\in \L$$
and satisfy
$$
|A|_{s,\b} := \sup_{a,b\in \L}(w_aw_b)^\beta\left\|A_{[a]}^{[b]}\right\| \Big(\frac{w(a,b)+|w^2_a-w^2_b|}{w(a,b)}\Big)^{s/2}<\infty$$
where $A_{[a]}^{[b]}$ denotes the restriction of $A$ to the block $[a]\times[b]$, $w(a,b) = \min(w_{a},w_{b})$ and $\|\cdot\|$ denotes the operator norm induced by the $Y_{0}$-norm.

\noindent {\bf A class of regularizing Hamiltonian functions.}
Let us fix any $n\in\N$. 
On the space
$$
\C^n\times \C^n \times Y_s$$
we define the norm
$$\|(z,r,\zeta)\|_s=\max( |z|, |r|, \|\zeta\|_s).$$
For $\sigma>0$ we denote
$$
\T^n_\sigma=\{z\in\C^n: | \Im z|<\sigma\}/2\pi\Z^n.
$$
%
For $\sigma,\mu\in(0,1]$ and $s\ge0$ we set
$$
\O^s(\s,\mu)=\T^n_\s\times \{r\in\C^n: |r|<\mu^2\}\times \{\zeta\in Y_s: \|\zeta\|_s<\mu\}
$$
We will denote points in $\O^s(\s,\mu)$ as $x=(\theta,r,\zeta)$.
A
 function defined on a domain $\O^s(\s,\mu)$, is called {\it real} if it gives real values to real
arguments.\\
Let 
$$
\D=\{\rho\}\subset \R^p 
$$ 
be a compact set of positive Lebesgue  measure. This is the set
of parameters upon which will depend our objects. Differentiability of functions on $\D$
is understood in the sense of Whitney. So $f\in C^1(\D)$ if it may be extended to a $C^1$-smooth
function $\tilde f$ on $ \R^p$, and $|f|_{ C^1(\D)}$ is the infimum of  $|\tilde f|_{ C^1(\R^p)}$,
taken over all $C^1$-extensions $\tilde f$ of $f$. \\
If $(z,r,\zeta)$ are $C^1$ functions on $\D$, then we define
$$\|(z,r,\zeta)\|_{s,\D}=\max_{j=0,1}( |\partial^j_\r  z|, |\partial^j_\r  r|, \|\partial^j_\r \zeta\|_s).$$
Let 
 $f:\O^0(\s,\mu)\times\D\to \C$ be a $C^1$-function, 
 real holomorphic in the first variable $x$,
  such that for all $\rho\in\D$
 $$
 \O^{s}(\s,\mu)\ni x\mapsto \nabla_\zeta f(x,\rho)\in Y_{s+\b}$$
and
 $$
 \O^{s}(\s,\mu)\ni x\mapsto\nabla^2_{\zeta} f(x,\rho)\in \M_{s,\beta}$$
are real holomorphic functions. We denote this set of functions by 
  $\Tc^{s,\beta}(\s,\mu,\D)$. We notice that for $\b>0$, both the gradient and the hessian of $f\in\Tc^{s,\beta}(\s,\mu,\D)$ have a regularizing effect.\\
 For a function $f\in \Tc^{s,\beta}(\s,\mu,\D)$ we define
 the norm 
 $$[f]^{s,\beta}_{\s,\mu,\D}$$
 through 
$$\sup
\max( |\partial^j_\r f(x,\r)|,\mu \|\partial^j_\r \nabla_\zeta f(x,\r)\|_{s+\b},
\mu^2|\partial^j_\r \nabla^2_\zeta f(x,\r)|_{s,\beta}),
$$
where the supremum is taken over all
$$
j=0,1,\ x\in \O^{s}(\s,\mu),\ \rho\in\D.$$
In the case $\b=0$ we denote $\Tc^{s}(\s,\mu,\D)=\Tc^{s,0}(\s,\mu,\D)$ and
$$[f]^{s}_{\s,\mu,\D}=[f]^{s,0}_{\s,\mu,\D}.$$

\noindent{\bf Normal form:} 
We introduce the orthogonal projection $\Pi$ defined on the $2\times 2$ complex matrices 
$$\Pi: \M_{2\times 2}(\C)\to \C I+\C J$$
where 
$$I=\left(\begin{array}{cc} 1&0\\0&1\end{array}\right) \quad \text{and}\quad J=\left(\begin{array}{cc} 0&-1\\1&0\end{array}\right).$$
\begin{definition}
 A matrix $A:\ \L\times \L\to \M_{2\times 2}(\C)$ is on normal form  and we denote $A\in  \NF$ if
 \begin{itemize}
 \item[(i)] $A$ is real valued,
 \item[(ii)] $A$ is symmetric, i.e. $A_b^a={}^t\hspace{-0,1cm}A_a^b$,
 \item[(iii)] $A$ satisfies $\Pi A=A$,
 \item[(iii)] $A$ is block diagonal, i.e. $A_b^a=0$ for all $w_a\neq w_b$.
 \end{itemize}
 \end{definition}
To a real symmetric matrix $A=(A_a^b)\in \M$  we associate in a unique way a real quadratic form on $Y_s\ni (\zeta_a)_{a\in \L}=(p_a,q_a)_{a\in\L}$ 
$$q(\zeta)=\frac 1 2 \sum_{a,b\in\L}\langle \zeta_a,\ A_a^b\zeta_b\rangle.$$
In the complex variables, $z_a=(\xi_a,\eta_a),\ a\in \L$, where
$$
\xi_a=\frac 1 {\sqrt 2} (p_a+iq_a),\quad \eta_a =\frac 1 {\sqrt 2} (p_a-iq_a),$$
we have
$$
q(\zeta)=\frac 1 2\langle \xi,\nabla_\xi^2 q\ \xi\rangle+ \frac 1 2\langle \eta,\nabla_\eta^2 q\ \eta\rangle+\langle \xi,\nabla_{\xi}\nabla_{\eta} q\ \eta\rangle    .$$
The matrices $\nabla_\xi^2 q$ and $\nabla_\eta^2 q$ are symmetric and complex conjugate of each other while $\nabla_{\xi}\nabla_{\eta} q$ is Hermitian. If $A\in \msb$ then
\be\label{q}
\sup_{a,b}\big\|(\nabla_\xi\nabla_\eta q)_{[a]}^{[b]}\big\| \leq\frac{|A|_{s,\b}}{(w_aw_b)^{\beta}\left({1+|w_a-w_b|}\right)^s}.\ee
We note that if $A$ is on normal form, then the associated  quadratic form $q(\zeta)=\frac 1 2 \langle \zeta,A\zeta\rangle$ reads in complex variables
\be\label{Q} q(\zeta)=\langle \xi,Q \eta\rangle \ee
where $Q:\L\times\L\to \C$ is 
\begin{itemize}
\item[(i)] Hermitian, i.e. $Q_b^a=\overline{Q_a^b}$,
\item[(ii)] Block-diagonal. 
\end{itemize}
In other words, when $A$ is on normal form, the associated quadatic form reads
$$q(\zeta)=\frac 1 2 \langle p,A_1 p\rangle+ \langle p,A_2 q\rangle+\frac 1 2 \langle p,A_1 q\rangle$$
with $Q=A_1+iA_2$ Hermitian.\\
 By extension we will say that a Hamiltonian is on normal form if it reads
\be\label{h}
h=\lan\om(\r), r \ran+\frac 1 2\langle \zeta, A(\r)\zeta\rangle\ee
with $\om(\r)\in\R^n$ a frequency vector and $A(\r)$ on normal form for all $\r$.

\smallskip

\subsection{Hypothesis on the spectrum of $A_0$.} We assume that  $A_0$ is a real diagonal matrix whose diagonal elements $\la>0, \ a\in\L$ are $C^1$. Our hypothesis depend on two constants $1>\delta_0>0$ and $c_0>0$ fixed once for all.

\smallskip

\noindent {\bf Hypothesis A1 -- Asymptotics.}
We assume that there exist $\ga\geq1$   such that
\be\label{laequiv}
c_{0}\, w_{a}^{\ga}\leq \la \leq  \frac{1}{c_0} w_a^{\ga} \quad \mbox{  for  } \r\in\D \text{ and } a\in\L
\ee
and
\be\label{la-lb}
|\la -\lb |\geq {c_0}{|w_a-w_b|} \quad \text{for }a,b\in\L\,. 
\ee

\smallskip

\noindent
{\bf Hypothesis A2 -- Non resonances.}
There exists a $\delta_0>0$ such that for all $\Ca^1$-functions
$$\omega:\D\to \R^n,\quad |\omega-\omega_0|_{\Ca^1(\D)}<\delta_0,$$
the following holds for each  $k\in\Z^n\setminus 0$: either we have the following properties :
$$
\left\{\begin{array}{cc} | \langle k,\om(\r)\rangle|\geq \delta_0& \mbox{for all}\;  \r\in\D,\\ 
 | \langle k,\om(\r)\rangle+\la|\geq \delta_0 w_a& \mbox{ for all} \; \r\in\D \; \mbox{and}\; a\in\L,\\
 |  \langle k,\om(\r)\rangle+\la+\lb|\geq \delta_0(w_a+w_b)& \mbox{ for all}\;  \r\in\D \;\mbox{and} a,\ b\in\L,\\
 |  \langle k,\om(\r)\rangle+\la-\lb|\geq \delta_0 (1+|w_a-w_b|)& \mbox{ for all}\;  \r\in\D\;\mbox{ and} \;a,\ b\in\L,\end{array}\right. $$
 or  there exists a unit vector ${\mathfrak z}\in\R^p$  such that 
$$ \nazz  (\langle k,\omega\rangle)  \geq \delta_0$$ for all  $\r\in\D$.  The first term of the alternative will be used in order to control the small divisors for large $k$, and the second one is featured  to control them for small $k$.

%
%
%

The last assumption above will be used to bound from below divisors $|  \langle k,\om(\r)\rangle+\la(\r)-\lb(\r)|$
with $w_a,\, w_b\sim1$. To control the (infinitely many) divisors with $\max(w_a,w_b)\gg1$ we need another assumption:
\medskip

\noindent
{\bf Hypothesis A3 -- Second Melnikov condition in measure.} There exist absolute constants $\a_1>0$, $\a_2>0$ and $C>0$ such that
for all $\Ca^1$-functions
$$\omega:\D\to \R^n,\quad |\omega-\omega_0|_{\Ca^1(\D)}<\delta_0,$$
the following holds:\\
for each $\ka>0$ and $N\ge1$ there exists a  closed subset $\D'=\D'(\om_0,\kappa,N)\subset \D$  satisfying 
\be\label{mesomega2}
\meas(\D\setminus {\D'})\leq CN^{\a_1}  (\frac{\ka}{ \delta_0} )^{\a_2}\quad (\a_1,\a_2\ge 0)  \ee
 such that for all $\r\in{\D'}$,  all $0<|k|\leq N$ and  all $a,b\in\L$ 
 we have
 \be\label{D33}
| \langle k,\om(\r)\rangle +\la-\lb|\geq \ka(1+|w_a-w_b|).
\ee

\subsection{The abstract KAM Theorem.}
We are now in position to state our abstract KAM result.
\begin{theorem}\label{main} Assume that 
\be \label{h0}h_0=\lan\om_0(\r), r\ran +\frac 1 2 \langle \zeta,  A_0\zeta\rangle\ee  with the spectrum of $A_0$ satisfying Hypothesis A1, A2, A3 and let $f\in\Tc^{s,\beta}(\D,\s,\mu)$ with $\b>0$, $s>0$. 
There exists $\eps_0>0$ (depending on $n,d,s,\b,\s,\mu$, on $\A$, $c_{0}$ and $\sup|\nabla_{\rho}\omega|$), $\a>0$  (depending on $n$, $d^*$, $s$, $\b$, $\a_1$, $\a_2$) and $\upsilon(\b,d^{*})>0$ such that\footnote{An explicit choice of $\upsilon$ is given in \eqref{ed} but is surely far from optimality.} if 
$$[f]^{s,\beta}_{\s,\mu,\D}=\eps<\min \left(\eps_0 , \delta_0^{\upsilon(\b,d^{*})}\right)$$
there is a $\D'\subset \D$ with $\text{meas}(\D\setminus \D')\leq \eps^\a$ such that for all $\r\in \D'$ the following holds: There are a real analytic symplectic diffeomorphism
$$\Phi:\O^s(\s/2,\mu/2)\to \O^s(\s,\mu)$$
and a vector $\om=\om(\r)$ such that 
$$
(h_{0}+  f)\circ \Phi= \lan\om(\r), r\ran +\frac 1 2 \langle \zeta,  A(\r)\zeta\rangle + \tilde f(r,\theta,\zeta;\r)$$
where $\partial_\zeta \tilde f=\partial_r \tilde f=\partial^2_{\zeta\zeta}\tilde f=0$ for $\zeta=r=0$ and $ A:\L\times\L \to \M_{2\times 2}(\R)$  is on normal form, i.e. $ A$ is real symmetric and block diagonal: $ A_{a}^{b}=0$ for all $w_a\neq w_b$.\\
Moreover $\Phi$ satisfies
$$
\|\Phi-\mathrm{Id}\|_s\leq  \eps^{11/12}\,,$$
for all $(r,\theta,\zeta)\in \O^s(\s/2,\mu/2)$, and
\begin{align*} &\left|  A(\r)-A_0\right|_\b\leq  \eps, \\
&| \om(\r)-\om_0(\r)|_{C^1}\leq  \eps
\end{align*}
 for all $\r\in \D'$.
 \end{theorem}

This  normal form result has dynamical consequences.  For $\r\in\D'$, the torus $\{0\}\times\T^n\times\{0\}$ is invariant by the flow of $(h_0+f)\circ \Phi$ and the dynamics of the Hamiltonian vector field of $h_0+f$ on the $\Phi(\{0\}\times\T^n\times\{0\})$ is the same as that of $$\lan\om(\r), r\ran +\frac 1 2 \langle \zeta,  A(\r)\zeta\rangle .$$
The Hamiltonian vector field on the torus $\{\zeta=r=0\}$ is
$$
\left\{\begin{array}{l}
\dot\zeta=0\\
\dot \theta=\om\\
\dot r=0,
\end{array}\right.
$$
and  the flow on the torus is linear: $t\mapsto \theta(t)=\theta_0+t\om$.\\
Moreover, the linearized equation on this torus reads 
$$
\left\{\begin{array}{l}
\dot\zeta=JA\zeta+J\partial^2_{r\zeta}f(0,\theta_0+\om t,0)\cdot r\\
\dot\theta=\partial^2_{r\zeta}f(0,\theta_0+\om t,0)\cdot \zeta+\partial^2_{rr}f(0,\theta_0+\om t,0)\cdot r\\
\dot r=0.
\end{array}\right.
$$
Since $A$ is on normal form (and in particular real symmetric and block diagonal) the eigenvalues of the $\zeta$-linear part are purely imaginary: $\pm i \tilde\lambda_a,\ a\in\L$. Therefore the invariant torus is linearly stable in the classical sense (all the eigenvalues of the linearized system are purely imaginary). Furthermore if the $\tilde\lambda_a$ are non-resonant with
respect to the frequency vector  $\om$ (a property  which can be guaranteed restricting   the
set $\D'$ arbitrarily little) 
then the linearized equation is reducible to constant coefficients.  Then the $\zeta$-component (and of
course also the $r$-component) will have only quasi-periodic
(in particular bounded) solutions. 

\section{Applications to Klein Gordon on $\S^d$}
 In this section we prove Theorem \ref{thmKG} as a corollary of Theorem \ref{main}. We use notations introduced in the introduction (see in particular \eqref{nota}).
Then the Klein Gordon Hamiltonian $H$ reads (up to a constant)
$$H(r,\theta,\zeta)= \lan\om_0(\r),r\ran+\sum_{a\in\L} \la\xi_a\eta_a+ \eps f(r,\theta,\zeta)$$
where 
$$f(r,\theta,\zeta)=
\int_{\S^d}G\left(x, \hat u(r,\theta,\zeta)(x)\right)dx .$$

\begin{lemma}\label{AOK2}
Hypothesis A1, A2 and A3 hold true with $\D=[1,2]^n$ and  
\be\label{d0}\de_0=\Big(\frac{\de}{2\sqrt{2+d+m}\max({w_a,\ a\in\A})}\Big)^3.\ee \end{lemma}
\proof
Hypothesis A1 is clearly satisfied with $c_0=1/2$ and $\ga=1$. The control of the cardinality of the clusters \eqref{block} is given with $d^{*}=d-1$.\\
On the other hand choosing $\mathfrak z\equiv z_k=\frac k{|k|}$ we have
\be\label{ouf} \nazz  (\langle k,\omega\rangle) \geq \frac{\de}{2\max({(\omega_{0})_a,\ a\in\A})} |k|\geq \frac{\de}{\sqrt{2+d+m}\max({w_a,\ a\in\A})} |k| \quad \text{ for all }k\neq 0\ee
 while 
\be\label{ouf2} \nazz  \la =0\quad \text{ for all } a\in\L .\ee
Then for all $k\neq 0$ the second part of the alternative in Hypothesis A2 is satisfied choosing $$\de_0\leq \de_* :=\frac{\de}{2\max({w_a,\ a\in\A})\sqrt{2+d+m}}.$$
It remains to verify A3. Without loss of generality we can assume $w_a\leq w_b.$\\
First  denoting
$$F_\ka(k,a,b):=\{\r\in\D\mid |\lan \om,k\ran +\la-\lb|\leq \ka\},$$
we have using \eqref{ouf} that
$$\meas F_\ka(k,a,b)\leq C(k,a,b) \frac{\ka}{\delta_*}.$$
On the other hand, defining
$$G_\nu(k,e):=\{\r\in\D\mid |\lan \om,k\ran+e|\leq 2\nu\},$$
we have, using again \eqref{ouf} that
$$\meas G_\nu(k,e)\leq C \frac{\nu}{\delta_*}.$$
Further $|\lan \om,k\ran+e|\leq 1$ can occur only if $|e|\leq C|k|$ and thus
$$G_\nu=\bigcup_{\substack{{0<|k|\leq N}\\{e\in\Z}}} G_\nu(k,e)$$
has a Lebesgue measure less than $CN^{n+1} \frac{\nu}{\delta_*}.$\\
Now we remark that 
$$|j+\frac {d-1}{2}-\sqrt{j(j+d-1)+m}|\leq \frac {C'_{m,d}}{j}$$
where $C'_{m,d}$ only depends on $m$ and $d$, from which we deduce
$$|\la-\lb -(w_a-w_b)|\leq \frac {2C'_{m,d}}{w_a}.$$
Therefore for $\r\in \D\setminus G_\nu$ and $w_a\geq \frac {2C'_{m,d}} \nu$ we have for all  $0<|k|\leq N$ 
$$|\lan \om,k\ran +\la-\lb|\geq \nu.$$
Finally  $w_a\leq \frac {2C'_{m,d}} \nu$ and $|\lan \om,k\ran +\la-\lb|\leq 1$ leads to $w_b\leq \frac {2C'_{m,d}} \nu+CN$ and thus,
if we restrict $\r$ to 
$$\D'=\D\setminus \Big[G_\nu \cup \Big(  \bigcup_{\substack{{0<|k|\leq N}\\{w_a,w_b\leq \frac {2C'_{m,d}} \nu+CN}}}F_\ka(k,a,b) \Big)\Big]$$
we get
$$|\lan \om,k\ran +\la-\lb|\geq \min(\ka,\nu), \quad 0<|k|\leq N,\ a,b\in\L.$$
Further
$$\meas \D\setminus \D'\leq CN^{n+1} \frac{\nu}{\delta_*}+\left(\frac {2C'_{m,d}} \nu+CN \right)^2N^n\frac{\ka}{\delta_*}.$$
Then choosing $\nu= \ka ^{1/3}$ and $\de_0=\de_*^{3}$, this measure is controlled by
$$CN^{n+2}\big(\frac \ka {\de_0}\big)^{1/3}$$
and we have
$$|\lan \om,k\ran +\la-\lb|\geq\ka, \quad \text{for } \r\in\D',\  0<|k|\leq N \text{ and } a,b\in\L.$$
Now we remark that for $|\la-\lb|\geq 2|\lan \om,k\ran|$, 
$$|\lan \om,k\ran +\la-\lb|\geq \frac 12 |\la-\lb|\geq \frac 14 (1+|w_a-w_b|)\geq \ka (1+|w_a-w_b|)$$
if we assume $\ka\leq \frac 14$. \\
On the other hand, when $|\la-\lb|\leq 2|\lan \om,k\ran|\leq CN$, 
$$|\lan \om,k\ran +\la-\lb|\geq  \tilde\ka(1+|w_a-w_b|)$$
where $\tilde\ka=\frac\ka {1+CN}$. Thus we get 
$$|\lan \om,k\ran +\la-\lb|\geq\tilde\ka(1+|w_a-w_b|), \quad \text{for } \r\in\D',\  0<|k|\leq N \text{ and } a,b\in\L$$
with
$$\meas  \left(\D\setminus \D'\right)\leq CN^{n+3}\big(\frac {\tilde\ka} {\de_0}\big)^{1/3}.$$
  \endproof
\begin{lemma}\label{FOK2}
Assume that $(x,u)\mapsto g(x,u) $ is real analytic on $\S^d\times\R$ and $s>1$ then there exist $\s>0$, $\mu>0$ such that the maping 
$$\O^s(\s,\mu)\times\D\ni(r,\theta,\zeta;\r)\mapsto f(r,\theta,\zeta;\r):=\int_{S^d}G(x,\hat u(r,\theta,\zeta)(x)) dx,$$
where $\hat u$ is defined in \eqref{uhat}, belongs to $\Tc^{s,1/2}(\s,\mu,\D)$ for any $s$ of the form $2N-\frac12$ with $N\in\N$ and $N>d$.
\end{lemma}
\proof First we notice that $f$ does not depend on the parameter $\r$.
Due to  the analyticity of $g$ and the fact that\footnote{$s>d/2$ is needed to insure that $Y_s$ is an algebra.} $s>d/2$, there exist positive
 $\s$ and $\mu$  such that $f:\O(\s,\mu)\times\D\to \C$ is  a $C^1$-function, analytic in the first variables $(r,\theta,\zeta)$, whose gradient in $\zeta$ analytically 
maps $Y_{s}$ to $Y_{-s}$. Further we have
$$
\frac{\partial f}{\partial\xi_a}=\frac{\partial f}{\partial\eta_a}=\frac { 1}{2\la^{1/2}}\int_{\S^d}G(x,\hat u(x))\Psi_a(x)\, dx.
$$ 
Since $x\mapsto G(x,\hat u(x))\in H^s(\S^d)$, we deduce that $\nabla_\zeta f \in Y_{s+1/2}$.\\
It remains to verify that $\nabla^2_\zeta f(r,\theta,\zeta;\r)\in\M_{s,1/2}$.\\
We have
\be
\frac{\partial^2 f}{\partial\xi_a\xi_b}=\frac{\partial^2 f}{\partial\eta_a\eta_b}=\frac{\partial^2 f}{\partial\xi_a\eta_b}=\frac { 1}{2\la^{1/2}\lb^{1/2}}\int_{\S^d}G(x,\hat u(x))\Psi_a\Psi_b \, dx.
\ee
 We note that for $s>d/2$ and $(r,\theta,\zeta)\in\O^s(\s,\mu)$, $x\mapsto \hat u(x)$ is bounded on $\S^d$.\\
It remains to prove that  the infinite matrix $M$ defined by
$$M_a^b= \frac { 1}{\la^{1/2}\lb^{1/2}}\int_{\S^d}G(x,\hat u)\Psi_a\Psi_b \, dx$$
belongs to $\M_{s,1/2}$, i.e.
$$\sup_{a,b\in\L}w_a^{1/2}w_b^{1/2}\left\| M_{[a]}^{[b]}\right\|\Big(\frac{w(a,b)+|w^2_a-w^2_b|}{w(a,b)}\Big)^{s/2+1/4}<\infty$$
where we recall that $w(a,b)=\min(w_a,w_b)$. The case $w_{a}=w_{b}$ is straightforward, since $\lambda_{a} \sim w_{a}$, $x\mapsto \hat u(x)$ is bounded on $\S^d$, and $\Phi_{a}$, $\Phi_{b}$ are normalized in $L^{2}(\S^{d})$. \\
If $w_{a}\neq w_{b}$, first we notice that
$$\left\| M_{[a]}^{[b]}\right\|=\sup_{\|u\|,\|v\|=1}|\langle M_{[a]}^{[b]}u,v  \rangle|=\frac { 1}{\la^{1/2}\lb^{1/2}}\sup_{\substack{\Phi_a\in E_{[a]},\ \|\Phi_a\|=1\\ \Phi_b\in E_{[b]},\ \|\Phi_b\|=1}}\Big|\int_{\S^d}G(x,\hat u)\Phi_a\Phi_b \, dx\Big|,$$
where $E_{[a]}$ (resp. $E_{[b]}$) is the eigenspace of $-\Delta$ associated to the cluster $[a]$ (resp. $[b]$).   
Then we follow arguments developed in \cite[Proposition 2]{Bam07}. The basic idea lies in the following commutator lemma: Let $A$ be a linear operator which maps $H^s(S^d)$ into itself and define the sequence of operators
$$
A_N:=[-\Delta,A_{N-1}], \quad A_0:=A$$
where $\Delta$ denotes the Laplace Beltrami operator on $\S^d$, then with \cite[Lemma 7]{Bam07}, we have for any $a,b\in\L$ with $w_a\neq w_b$ and any $N\geq 0$
$$
|{\langle A\Phi_a,\Phi_b\rangle}|\leq \frac{1}{|w^2_a-w^2_b|^N}|{\langle A_N\Phi_a,\Phi_b\rangle}|.$$
Let $A$ be the  operator  given by the multiplication by the function $$\Phi(x)=G(x,\hat u(r,\theta,\zeta)(x)).$$ We note that $\Phi\in H^{s+1/2}$ for $(r,\theta,\zeta)\in \O^s(\s,\mu)$.  Then, by an induction argument,
$$
A_N=\sum_{0\leq |{\alpha}|\leq N}C_{\alpha,N}D^\alpha$$
where
$$
C_{\alpha,N}= \sum_{0\leq |{\beta}|\leq 2N-|{\alpha}|} V_{\alpha,\beta,N}(x) D^\beta \Phi$$
and $V_{\alpha,\beta,N}$ are $C^\infty$ functions (cf. \cite[Lemma 8]{Bam07}). Therefore one gets
\begin{align*}\begin{split}\label{3.5}
|{\int_{S^d}\Phi_a\Phi_b\Phi dx}|&\leq  \frac{1}{|w_a^2-w_b^2|^N}\norma{A_N \Phi_a}_{L^2}\\
&\leq C   \frac{1}{|w_a^2-w_b^2|^N} 
\sum_{0\leq |{\alpha}| \leq  N} \sum_{0\leq |{\beta}| \leq 2N-|{\alpha}|} ||  D^\beta \Phi D^\alpha \Psi_a ||_{L^2}  \\
& \leq C   \frac{1}{|w_a^2-w_b^2|^N} 
 \Big(\sum_{0\leq |{\alpha}| \leq  N/2}\sum_{0\leq |{\beta}| \leq 2N-|{\alpha}|}  \norma{\Phi_a}_{|\alpha|+\nu_0} ||\Phi ||_{|\beta|} \\
&+ \sum_{N/2< |{\alpha}| \leq  N} \sum_{0\leq |{\beta}| \leq 2N-|{\alpha}|}  \norma{\Phi_a}_{|\alpha|} ||\Phi ||_{|\beta|+\nu_0} \Big) \\
&  \leq C \frac{1}{|w_a^2-w_b^2|^N}  \norma{\Phi_a}_{N} ||\Phi ||_{{2N}}                    
\end{split}\end{align*}
where we used 
$$\forall \nu_0 > d/2 \qquad \|fg\|_{L^2} \leq C \|f\|_{{\nu_0}} \|g\|_{L^2} .            $$
On the other hand since $-\Delta \Phi_a= w_a(w_a+d-1)\Phi_a$
\be\label{phi-s}
\norma{\Phi_a}_N\leq Cw_a^{N}.\ee
Therefore choosing $N=\frac12 (s+\frac12)$
\begin{align*}
|{\int_{S^d}\Phi_a\Phi_b\Phi dx}|&\leq  C \Big(\frac{w_a}{|w_a^2-w_b^2|}\Big)^{s/2+1/4}  \\
&\leq  2^{s/2+1/4}C \Big(\frac{w_a}{w(a,b)+|w_a^2-w_b^2|}\Big)^{s/2+1/4} .                   
\end{align*}
Clearly the same estimate remains true when interchanging $a$ and $b$. 
\endproof

So  Main Theorem applies (for any choice of  vector $I\in[1,2]^\A$) and Theorem \ref{thmKG} is proved.
 
\begin{remark}\label{rem-Zoll}
Theorem \ref{thmKG} still holds true when we consider the Klein Gordon equation on a Zoll manifold. This technical extension follows from  results and computations in \cite{DS1} and \cite{BDGS}. We prefer to focus on the sphere in order to simplify the  presentation.
\end{remark}
\begin{remark}\label{rem-beam} We can also consider the Beam equation on the torus $\T^d$  with convolution potential in a Sobolev-like phase space:
\be \label{beam}u_{tt}+\Delta^2 u+m u+V\star u + \eps \partial_u G(x,u)=0 ,\quad   x\in \T^d.
\ee
 Here $m$ is the mass,  $G$ is a real analytic function on $\T^d\times \R$ vanishing at least of order 3 at the origin. The convolution potential $V:\ \T^d\to \R$ is supposed to be analytic with
   real  positive Fourier coefficients $\hat V(a)$, $a\in\Z^d$. The same equation, but in an analytic phase space, were considered in \cite{EGK1,EGK2}. Actually following \cite{EGK1} and the proof of Lemma \ref{FOK2}, in order to apply our abstract KAM theorem, it remains to control the $|\cdot|_{s,1/2}$-norm of the infinite matrix\footnote{Here $\la=\sqrt{|a|^4+m}$ and $\Psi_a(x)=e^{ia\cdot x}$, $a\in\Z^d$.}
$$M_a^b= \frac { 1}{\la^{1/2}\lb^{1/2}}\int_{\T^d}\partial^2_uG(x,u)\Psi_a\Psi_b \, dx$$
restricted to the block defined by $[a]=\{b\in\Z^d\mid |a|=|b|\}$. This is achieved in the same way as in Lemma \ref{FOK2}. 
\end{remark}
\begin{remark}\label{rem-NLW} 
Notice that our theorem does not apply to the nonlinear wave equation:
\be \label{NLW}u_{tt}+\Delta u+m u+V\star u + \eps \partial_u G(x,u)=0 ,\quad   x\in \T^d
\ee
since in that case the second Melnikov condition is not satisfied.
\end{remark}

\section{Poisson brackets and Hamiltonian flows.}
It turns out that the  space $\Tc^{s,\beta}(\s,\mu,\D)$ is not stable by Poisson brackets.
Therefore, in this section, we first define a new  space $\Tc^{s,\beta+}(\s,\mu,\D)\subset \Tc^{s,\beta}(\s,\mu,\D)$ and then we prove a structural stability  which is essentially contained in the claim
 $$\{\Tc^{s,\beta+}(\s,\mu,\D)\ ,\ \Tc^{s,\beta}(\s,\mu,\D)\} \in \Tc^{s,\beta}(\s,\mu,\D).$$  
We will also study the hamiltonian flows generated by hamiltonian functions in $\Tc^{s,\beta+}(\s,\mu,\D)$. In this section, all constants $C$ will depend only on $s$, $\beta$ and $n$.

\subsection{New Hamiltonian space}
We introduce $\Tc^{s,\beta+}(\s,\mu,\D)$ defined by
$$\Tc^{s,\beta+}(\s,\mu,\D)=\{f\in \Tc^{s,\beta}(\s,\mu,\D)\mid\ \partial^j_\r \nabla^2_\zeta f\in\msb^+,\ j=0,1\}$$
where 
$$\msb^+=\{M\in\msb\mid |M|_{s,\b+} <\infty \} $$
and 
\begin{align*}|M|_{s,\b+}=\sup_{a,b\in\L}(1+|w_a-w_b|)\Big(\frac{w(a,b)+|w^2_a-w^2_b|}{w(a,b)}\Big)^{\frac{s}{2}}(w_aw_b)^\beta\left\|M_{[a]}^{[b]}\right\| .\end{align*}
We endow $\Tc^{s,\beta+}(\s,\mu,\D)$ with the norm
$$[f]_{\s,\mu,\D}^{s,\b+}=[f]_{\s,\mu,\D}^{s,\b}+\sup_{j=0,1}\Big(\mu^2|\partial^j_\r \nabla^2_\zeta f|_{s,\b+}\Big).$$

\begin{lemma}\label{product} Let $0<\b\leq 1$ and $s>d/2$ there exists a constant $C\equiv C(\b,s)>0$ such that
\begin{itemize}
\item[(i)]
Let $A\in \msb$ and $B\in \msb^+$ then $AB$ and $BA$ belong to $\msb$ and
$$|AB|_{s,\b},\ |BA|_{s,\b}\leq C|A|_{s,\b}|B|_{s,\b+}.$$
\item[(ii)]
Let $A,B\in \msb^+$  then $AB$ and $BA$ belong to $\msb^+$ and
$$|AB|_{s,\b+},\ |BA|_{s,\b+}\leq C|A|_{s,\b+}|B|_{s,\b+}.$$
\item[(iii)] Let $A\in \msb^+$ then $A\in\L(Y_s,Y_{s+\b})$  and
$$\|A\zeta\|_{s+\b}\leq C|A|_{s,\b+}\|\zeta\|_s.$$
\item[(iv)] Let $X\in Y_s $ and $Y\in Y_s$  and  denote $A=X\otimes Y$ then $A$ and ${}^tA$ belong to $\L(Y_s)$ and
$$\|A\|_{\L(Y_s)},\|{}^tA\|_{\L(Y_s)}\leq C\|X\|_{s}\|Y\|_s.$$
\item[(v)] Let $X\in Y_{s+\b} $ and $Y\in Y_{s+\b}$  then $A=X\otimes Y\in\msb$ and
$$\|A\|_{s,\b}\leq C\|X\|_{s+\b}\|Y\|_{s+\b}.$$
\end{itemize}
\end{lemma}
\proof
(i) Let $a,b\in\L$
\begin{align*}
\left\| (AB)_{[a]}^{[b]} \right\|&\leq \sum_{c\in\Lc}\left\| A_{[a]}^{[c]} \right\|\left\| B_{[c]}^{[b]} \right\|\\
&\leq  \frac{|A|_{\b+}|B|_{\b}}{(w_aw_b)^{\b}}\Big(\frac{w(a,b)}{w(a,b)+|w_a^2-w_b^2|}\Big)^{\frac{s}2}
\sum_{c\in\Lc}\frac 1 {w_c^{2\b}(1+|w_a-w_c|)}\\
&\leq C \frac{|A|_{\b+}|B|_{\b}}{(w_aw_b)^{\b}}\Big(\frac{w(a,b)}{w(a,b)+|w_a^2-w_b^2|}\Big)^{\frac{s}2}
\end{align*}
where we used that by Lemma \ref{lem-A1}
$$\frac{w(a,b)}{w(a,b)+|w_a^2-w_b^2|}\geq  \frac{w(a,c)}{w(a,c)+|w_a^2-w_c^2|}\frac{w(c,b)}{w(c,b)+|w_c^2-w_b^2|}$$
 and that by Lemma \ref{lem-A2}, $\sum_{c\in\Lc}\frac 1 {w_c^{2\b}(1+|w_a-w_c|)}\leq C$ where $C$ only depends on $\b$.\\ 
(ii) Similarly let $a,b\in\L$ and assume without loss of generality that $w_a\leq w_b$
\begin{align*}
\left\| (AB)_{[a]}^{[b]} \right\|&\leq \sum_{c\in\Lc}\left\| A_{[a]}^{[c]} \right\|\left\| B_{[c]}^{[b]} \right\|\\
&\leq \frac{|A|_{\b+}|B|_{\b+}}{(w_aw_b)^{\b}}\Big(\frac{w(a,b)}{w(a,b)+|w_a^2-w_b^2|}\Big)^{\frac{s}2}
\sum_{c\in\Lc}\frac 1 {w_c^{2\b}(1+|w_a-w_c|)(1+|w_b-w_c|)}\\
&\leq  \frac{2|A|_{\b+}|B|_{\b+}}{(w_aw_b)^{\b}(1+|w_a-w_b|)}\Big(\frac{w(a,b)}{w(a,b)+|w_a^2-w_b^2|}\Big)^{\frac{s}2}\\
&\Big(\sum_{\substack{c\in\Lc \\ w_c\leq \frac12(w_a+w_b) }}\frac 1 {w_c^{2\b}(1+|w_a-w_c|)}
 +\sum_{\substack{c\in\Lc \\  w_c\geq \frac12(w_a+w_b) }}\frac 1 {w_c^{2\b}(1+|w_b-w_c|)}\Big)\\
&\leq C\frac{|A|_{\b+}|B|_{\b+}}{(w_aw_b)^{\b}(1+|w_a-w_b|)}\Big(\frac{w(a,b)}{w(a,b)+|w_a^2-w_b^2|}\Big)^{\frac{s}2}.
\end{align*}
(iii) Let $\zeta\in Y_s$  we have
\begin{align*}
\| A\zeta\|^2_{s+\b}&\leq\sum_{a\in\Lc}w_a^{2s+2\b}\big(\sum_{b\in\Lc}\|A_{[a]}^{[b]}\| \|\zeta_{[b]}\|\big)^2\\
&\leq |A|^2_{s,\b+}\sum_{a\in\Lc}\Big(\sum_{b\in\Lc}\frac{w_a^{s}\|w_b^s\zeta_{[b]}\|}{w_b^{s+\b}(1+|w_a-w_b|)}\Big(\frac{w(a,b)}{w(a,b)+|w_a^2-w_b^2|}\Big)^{\frac{s}2}\Big)^2\\
&\leq 2^{2s+1} |A|^2_{s,\b+}\sum_{a\in\Lc}\Big(\sum_{\substack{b\in\Lc \\ w_a\leq 2w_b) }}\frac{\|w_b^s\zeta_{[b]}\|}{w_b^\b(1+|w_a-w_b|)}\\
&\hspace{1cm}+\sum_{\substack{b\in\Lc \\ w_a\geq 2 w_b }} \frac{\|w_b^s\zeta_{[b]}\|w(a,b)^{\frac{s}2}}{w_b^{s+\b}(1+|w_a-w_b|)} \Big)^2\\
&\leq 2^{2s+1} |A|^2_{s,\b+}\sum_{a\in\Lc}\big(\sum_{b\in\Lc}\frac{\|w_b^s\zeta_{[b]}\|}{w_b^\b(1+|w_a-w_b|)}\big)^2\\
&\leq C|A|^2_{s,\b+}\|\zeta\|_s^2
\end{align*}
where we used that the convolution between the $\ell^{p}$ sequence, $p<2$, $\|w_b^{s-\beta}\zeta_{[b]}\|$ and the $\ell^{q}$ sequence, $q=\frac{2p}{3p-2}>1$, $\frac{1}{(1+|w_b|)}$ is a $\ell^{2}$ sequence in $a$ whose norm is bounded by $C\|\zeta\|_{s}$. \\
(iv) Let $u\in Y_s$, we have
\begin{align*}
\|Au\|_s&=|\langle Y,u\rangle| \|X\|_s\leq \|X\|_s\|Y\|_s\|u\|_s .
\end{align*}
(v) Let $a,b\in\L$
\begin{align*}
\left\| A_{[a]}^{[b]} \right\|&= \|X_{[a]}\| \|Y_{[b]}\|\leq (w_aw_b)^{-s-\b}\|X\|_{s+\b} \|Y\|_{s+\b}\\
&\leq (w_aw_b)^{-\b}\frac1{(1+|w^2_a-w^2_b|)^{s/2}}\|X\|_{s+\b} \|Y\|_{s+\b}\\ 
&\leq (w_aw_b)^{-\b}\Big(\frac{w(a,b)}{(w(a,b)+|w^2_a-w^2_b|)}\Big)^{s/2}\|X\|_{s+\b} \|Y\|_{s+\b}.
\end{align*}
\endproof

\subsection{ Jets of functions.} \label{ss5.1}

For any function $h\in \Tc^{s}(\s,\mu,\D)$ we define its jet $h^T=h^T(x,\r)$ as the following 
 Taylor polynomial of $h$ at $r=0$ and $\zeta=0$:
\be\begin{split}
\label{jet}
h^T=&h_\theta+\langle h_r, r\rangle+\langle h_\zeta,\zeta\rangle+\frac 1 2 \langle h_{\zeta\zeta}\zeta,\zeta \rangle\\
=&h(\theta,0,\r)+\langle \nabla_rh(\theta,0,\r),r \rangle+\langle \nabla_\zeta h(\theta,0,\r),\zeta\rangle+\frac 1 2 \langle \nabla^2_{\zeta \zeta}h(\theta,0,\r)\zeta,\zeta \rangle
\end{split}\ee
Functions of the form  $h^T$ will be called {\it jet-functions.}\\
Directly from  the definition of the norm 
$[h]^{s,\b }_{\s,\mu,\D}$ we get that 
\begin{align}\begin{split}\label{norm2}
&|h_\theta(\theta,\r)|\leq[h]^{s}_{\s,\mu,\D}, 
\quad |h_r(\theta,\r)|\leq\mu^{-2}[h]^{s}_{\s,\mu,\D},\\
&\|h_\zeta(\theta,\r)\|_{s+\b} \leq \mu^{-1}[h]^{s,\b}_{\s,\mu,\D},\quad
| h_{\zeta \zeta}(\theta,\r)|_{s,\b}\leq\mu^{-2}[h]^{s,\b}_{\s,\mu,\D},
\end{split}\end{align}
for any $\theta\in\T^n_\s$ and any $\r\in \D$. Moreover, the first derivative with respect to $\r$ will satisfy the same estimates.\\
We also notice that by Cauchy estimates we have that for $x\in\O(\s,\mu')$
\be\label{Cauchy}
\|\nabla^2_\zeta h(x)\|_{\L(Y_s,Y_{s+\b})}\leq \frac{\sup_{y\in\O(\s,\mu)}\|\nabla_\zeta h(y)\|_{s}}{\mu-\mu'}.\ee
Thus $h_{\zeta\zeta}$ is a linear continuous operator from $Y_s$ to $Y_{s+\b}$ and
\be \label{norm3}\|h_{\zeta\zeta}(\theta,\r)\|_{\L(Y_s,Y_{s+\b})}\leq \mu^{-2}[h]^{s}_{\s,\mu,\D} \ee
for any $\theta\in\T^n_\s$ and any $\r\in \D$. \\
\begin{proposition}\label{lemma:jet}
For any 
$h\in \Tc^{s,\b}(\s,\mu,\D)$  we have $h^T\in \Tc^{s,\b}(\s,\mu,\D)$,
$$
[h^T]^{s,\b}_{\s,\mu,\D}\leq C  [h]^{s,\b}_{\s,\mu,\D}\,,$$
and, for any   $0<\mu' < \mu$,
$$
[h-h^T]^{s,\b}_{\s,\mu',\D}
\leq  C \left(\frac{\mu'}{\mu}\right)^3 
 [h]^{s,\b}_{\s,\mu,\D}\,,
$$
where $C$ is an absolute constant. 
\end{proposition}
\proof 
We start with the second statement. Consider first the hessian  $\nabla^2_{\zeta\zeta}(h-h^T)(x)$ 
for $x=(\theta,r, \zeta)\in \O^{s}(\s,\mu')$.
 Let us denote $m=\mu'/\mu$. Then for $z\in\overline  D_1=\{z\in\C: |z|\le1\}$ 
 we have $(\theta, (z/m)^2 r,(z/m)\zeta)\in \O^{s}(\s,\mu)$. Consider the function
\begin{equation*}
\begin{split}
&f: D_1\times \O^{s}(\s,\mu')\to \M_\b\,,\\
& (z,x)\mapsto \nabla^2_{\zeta\zeta} h(\theta, (z/m)^2 r,(z/m)\zeta)=h_0(x)+h_1(x) z+\dots \,.
\end{split}
\end{equation*}
It 
is holomorphic and its norm is bounded by $\mu^{-2}[h]^{s,\b}_{\s,\mu,\D}$. So, by the 
Cauchy estimate, $|h_j (x)|_{s,\b}\leq \mu^{-2}[h]^{s,\b}_{\s,\mu,\D}$ for  $j=1,2,\dots$ and $x\in \O^{s}(\s,\mu')$.
Since $\nabla^2_{\zeta\zeta}(h-h^T)(x)=h_1(x)m+h_2(x)m^2+\cdots,$ then
$\nabla^2_{\zeta\zeta}(h-h^T)$ is holomorphic in $x\in  \O^{s}(\s,\mu)$, and 
$$
|\nabla^2_{\zeta\zeta}(h-h^T)(x)|_{s,\b}\leq \mu^{-2}[h]^{s,\b}_{\s,\mu,\D}(m+m^2+\dots)\leq \mu^{-2}[h]^{s,\b}_{\s,\mu,\D}\frac{m}{1-m}.
$$
So $\nabla^2_{\zeta\zeta}(h-h^T)$ satisfies the required estimate with $C=2$, if $\mu'\le \mu/2$. 

 Same argument applies to bound  the norms of $\partial_\r \nabla^2_{\zeta\zeta}(h-h^T)$, 
 $h-h^T$  and $ \nabla_\zeta(h-h^T)$ if $\mu'\le \mu/2$,
  and to prove the analyticity of these mappings.
 \smallskip
 
 Now we turn to the  first statement and  write $h^T$ as $h-(h-h^T)$. This implies that $h^T$, $\nabla_\zeta h^T$ and
 $\nabla^2_{\zeta \zeta} h^T$ are analytic on $\O^{s}(\s, \frac12 \mu)$ and that 
 $$
 [h^T]^{s,\b}_{\s, \frac12\mu,\D}\leq C_1 [h]^{s,\b}_{\s,\mu,\D}\,.
 $$
 Since  $h^T$ is a quadratic polynomial, then the mappings  $h^T$, $\nabla_\zeta h^T$ and
 $\nabla^2_{\zeta \zeta} h^T$ are as well  analytic on $\O^{s}(\s, \ \mu)$, and the
  norm  $ [h^T]^{s,\b}_{\s, \mu,\D}$
 satisfies the same estimate, modulo another constant factor, for any $0<\mu'\le\mu$.

 Finally, the estimate for $[h-h^T]^{s,\b}_{\s,\mu',\D}$ when $\mu/2\le\mu'\le\mu$, 
 with a suitable constant $C$, follows from the estimate for $[h^T]^{s,\b}_{\s,\mu,\D}$
 since $[h-h^T]^{s,\b}_{\s,\mu',\D}\le [h^T]^{s,\b}_{\s,\mu,\D} +
 [h]^{s,\b}_{\s,\mu,\D}\,$.
 \endproof

\subsection{Poisson brackets and flows}
The Poisson brackets of functions 
is defined by
\be\label{poisson}
\{f,g\}=\nabla_rf\cdot\nabla_\theta g- \nabla_\theta f\cdot\nabla_r g+\langle J\nabla_\zeta f,\nabla_\zeta g\rangle\,.
\ee

\begin{lemma}\label{lemma-poisson} Let $s\geq 1$.
Let $f\in\Tc^{s,\beta+}(\s,\mu,\D)$ and $g\in \Tc^{s,\beta}(\s,\mu,\D)$ be two jet functions then for any $0<\s'<\s$   we have $\{f,g\}\in \Tc^{s,\beta}(\s',\mu,\D)$ and
$$[\{f,g\}]_{\s',\mu,\D}^{s,\b}\leq C(\s-\s')^{-1}\mu^{-2}[f]_{\s,\mu,\D}^{s,\b+}[g]_{\s,\mu,\D}^{s,\b}.$$
\end{lemma}
\proof
Let denote by $h_1$, $h_2$, $h_3$ the three terms on the right hand side of \eqref{poisson}. Since $\nabla_rf(\theta,r,\zeta,\r)= f_r(\theta,\r)$ and $\nabla_rg(\theta,r,\zeta,\r)= g_r(\theta,\r)$ are independent of $r$ and $\zeta$,  the control of  $h_1$ and $h_2$ is straightforward by Cauchy estimates and \eqref{norm2}. \\
We focus on the third term in formula: $h_3=\langle J\nabla_\zeta f,\nabla_\zeta g\rangle$.
 As, from \eqref{jet}, we have $\nabla_\zeta f=f_\zeta +f_{\zeta\zeta}\zeta$ and similarly for $\nabla_\zeta g$, we obtain
$$
h_3= \langle  Jf_\zeta, g_\zeta \rangle-\langle  \zeta,f_{\zeta\zeta} J g_\zeta \rangle+\langle g_{\zeta\zeta} J f_\zeta, \zeta \rangle + \langle  g_{\zeta\zeta}J f_{\zeta\zeta}\zeta,\zeta \rangle.
$$
Using \eqref{norm2}, \eqref{norm3} and $\|\zeta\|_s\leq \mu$,
  we get
$$\ 
|h_3(x,\cdot)|\leq C\mu^{-2}[f]^{s,\b}_{\s,\mu,\D}[g]^{s,\b}_{\s,\mu,\D}\,,$$
for any $x\in\O(\s,\mu)$ and $\r\in\D$.\\ Since
$$
\nabla_\zeta h_3=-f_{\zeta\zeta}Jg_\zeta+g_{\zeta\zeta}J f_\zeta+g_{\zeta\zeta}Jf_{\zeta\zeta} \zeta-f_{\zeta\zeta}Jg_{\zeta\zeta} \zeta,$$
then, using \eqref{norm3} and  Lemma \ref{product}, we get that for  $x\in \O^{s}(\s,\mu)$  and $\r\in\D$
$$\|\nabla_\zeta h_3(x,\cdot)\|_{s+\b}\leq C\mu^{-3}[f]^{s,\b}_{\s,\mu,\D}[g]^{s,\b}_{\s,\mu,\D}\, .$$
Finally, as $\nabla^2h_3=g_{\zeta\zeta}Jf_{\zeta\zeta}-f_{\zeta\zeta}Jg_{\zeta\zeta} $,
 then, using again Lemma \ref{product} we get that for  $x\in \O^{s}(\s,\mu)$ and $\r\in\D$
   $$|\nabla^2h_3(x,\cdot)|_{s,\b}\leq C\mu^{-4}[f]^{s,\b+}_{\s,\mu,\D}[g]^{s,\b}_{\s,\mu,\D}\, .$$
\endproof

\subsection{Hamiltonian flows}

To any $C^1$-function $f$ on a domain
 $\O^s(\s,\mu)\times \D$ 
 we associate the Hamilton equations 
 \be\label{hameq} \left\{\begin{array}{lll}
\dot r = &\nabla_\theta f(r,\theta,\zeta;\r),\\
\dot \theta= &-\nabla_r f(r,\theta,\zeta;\r),\\
\dot \zeta= &J\nabla_\zeta f(r,\theta,\zeta;\r).
\end{array}\right.
\ee
and  denote by $\Phi^t_f\equiv \Phi^t$, $t\in\R$, the corresponding flow map (if it exists). 
Now let  $f\equiv f^T$ be  a jet-function
\be\label{f_eq}
f=f_\theta(\theta;\r)+f_r(\theta;\r)\cdot r+\langle f_\zeta(\theta;\r),\zeta\rangle+\frac 1 2 \langle f_{\zeta\zeta}(\theta;\r)\zeta,\zeta \rangle.
\ee
Then  Hamilton equations \eqref{hameq} take  the form\footnote{Here and below we often suppress the argument $\r$.}
\be\label{hameq-jet} \left\{\begin{array}{lll}
\dot r = &-\nabla_\theta f(r,\theta,\zeta),\\
\dot \theta= & f_r(\theta),\\
\dot \zeta= &J\left( f_\zeta (\theta)+f_{\zeta\zeta}(\theta)\zeta\right).
\end{array}\right.
\ee
Denote by $V_f=(V_f^r,V_f^\theta,V_f^\zeta)$ the corresponding vector field. It is analytic on any domain 
$\O^{s}(\sigma-2\eta,\mu-2\nu)=:\O_{2\eta, 2\nu}$, where $0<2\eta<\sigma$, $ 0< 2\nu<\mu$.  The
flow maps $\Phi^t_f$ of $V_f$ on $\O_{2\eta, 2\nu}$ are analytic as long as they exist. We will study them as long as they 
map  $\O_{2\eta, 2\nu}$  to  $\O_{\eta, \nu}$. \\
Assume that 
\be\label{hyp-f1}  
[f]^{s}_{\s,\mu,\D}\leq \frac 12\nu^2 \eta.
\ee
 Then for $x=(r,\theta,\zeta)\in\O_{2\eta, 2\nu}$ by the Cauchy estimate\footnote{Notice that the distance from $\O^{s}(\sigma-2\eta,\mu-2\nu)$ to $\p \O^{s}(\sigma,\mu)$ in the $r$-direction is $4\nu\mu-4\nu^2>4\nu^2$.} and
\eqref{norm3}  we have 
\begin{equation*}
\begin{split}
|V_f^r|_{\C^n} &\le (2\eta)^{-1}\ff\le\nu^2 ,\\
|V_f^\theta|_{\C^n} &\le (4\nu^2)^{-1} \ff\le \eta,\\
\|V_f^\zeta\|_{s} &\le \big(\mu^{-1}+\mu^{-2}\mu\big)\ff\le\nu.
\end{split}
\end{equation*}
Noting that the distance from $\O_{2\eta,2\nu}$ to $\p \O_{\eta,\nu}$ in the  $r$-direction is $2\nu\mu-3\nu^2>\nu^2$, in the $\theta$-direction is $\eta$ and in the $\zeta$-direction is
$\nu$, we see that the flow maps 
\be\label{flow}
\Phi^t_f:  \O^{s}(\sigma-2\eta,\mu-2\nu)\to \O^{s}(\sigma-\eta,\mu-\nu),\qquad 0\le t\le1,
\ee
are well defined and analytic. 

For  $x \in\O_{2s, 2\nu}$ denote $\Phi^t_f(x)=(r(t),\theta(t),\zeta(t))$. Since $V_f^\theta$ is independent from $r$ and 
$\zeta$, then $\theta(t)=K(\theta;t)$, where $K$ is analytic in both arguments.
As $V_f^\zeta=Jf_\zeta+Jf_{\zeta\zeta}\zeta$, where the non autonomous 
linear operator $Jf_{\zeta\zeta}(\theta(t))$ is bounded in the space $Y_{s}$ and both the operator and the curve
$Jf_\zeta(\theta(t))$ analytically depend on $\theta$ (through $\theta(t)=K(\theta;t)$), 
then $\zeta(t)=T(\theta,t)+U(\theta;t)\zeta$, where
$U(\theta;t)$ is a bounded linear operator, both $U$ and $T$ analytic in $\theta$.  Similar since $V_f^\zeta$ is a 
quadratic polynomial in $\zeta$ and an affine function of $r$, then $r(t)=L(\theta,\zeta;t)+S(\theta;t)r$, where $S$
is an $n\times n$ matrix and $L$ is a quadratic polynomial in $\zeta$, both analytic in  $\theta$.

The vector field $V_f$ is real for real arguments, and so behaves its flow map. Since the vector field is hamiltonian, 
then the flow maps are symplectic (e.g., see \cite{Kuk2}). We have proven

\begin{lemma}\label{l.flow} 
Let  $0<2\eta<\s$, $0<2\nu<\mu$ and 
 $f =f^T\in\Tc^s(\s,\mu,\D)$ satisfy 
\eqref{hyp-f1}. Then for   $0\le t\le1$ the
 flow maps  $\Phi^t_f$ of  equation \eqref{hameq-jet} define analytic mappings \eqref{flow}
and define
 symplectomorphisms from $\O^{s}(\s-2\eta,\mu-2\nu)$ to $\O^{s}(\s-\eta,\mu-\nu)$. 
 They have the form 
\be\label{Phi1} 
\Phi_f^t:
\left(\begin{array}{lll}r\\ \theta \\ \zeta \end{array}\right)
\to
\left(\begin{array}{lll}
L(\theta,\zeta;t)+S(  \theta;t)r\\
K( \theta;t)\\
T( \theta;t)+U( \theta;t)\zeta
\end{array}\right),
\ee
where  $L(\theta,\zeta;t)$ is  quadratic in $\zeta$,  while 
$U( \theta;t)$ and $S( \theta;t)$ are  bounded linear operators in corresponding spaces. 
\end{lemma}

Our next result specifies the flow maps $\Phi_f^t$ and their representation  \eqref{Phi1} when $f\in \Tbp$:

\begin{lemma}\label{changevar}  
Let  $0<2\eta<\s\leq 1$, $0<2\nu<\mu\leq 1$ and 
 $f =f^T\in\Tc^{s,\b+}(\s,\mu,\D)$ satisfy
\be\label{hyp-f}[f]^{s,\b+}_{\s,\mu,\D}\leq \frac 12 \nu^2 \eta
\ee  Then:\\
1)   Mapping
$L$ is analytic in $(\theta ,\zeta) \in  \T^{\sigma-2\eta}\times \O_\mu(Y_s)$. 
 Mappings $K,T$ and operators  $S$ and $U$ analytically depend on  $\theta\in \T^{\sigma-2\eta}$; their norms 
 and operator-norms  satisfy
\begin{equation}
\begin{split}\label{4.55}
\|S( \theta;t)\|_{\L(\C^n,\C^n)},  \|{}^tU( \theta;t)-I\|_{\L(Y_s, Y_{s+\b})},\\    \| U( \theta;t)-I\|_{\L(Y_s, Y_{s+\b})}, 
   | U( \theta;t)-I |_{s,\b+}\le2,
\end{split}\end{equation}
while for any component $L^j$ of $L$ and any $(\theta,r,\zeta)\in \O^{s}(\s-2\eta,\mu-2\nu)$  we have 
\begin{equation}
\begin{split}\label{4.56}
\|\nabla_\zeta L^j (\theta,\zeta;t) \|_{s+\b} &\le C \eta^{-1}\mu^{-1} \fbp ,\\
 |\nabla^2_\zeta L^j (\theta,\zeta;t) |_{s,\b+} &\le C \eta^{-1}\mu^{-2} \fbp.
\end{split}\end{equation}
2) 
The flow maps $\Phi^t_f$ analytically extend to  mappings \\  
$
\C^n\times\T^n_{\s-2\eta}\times Y_{s}\ni x^0= (r^0,\theta^0,\zeta^0)\mapsto x(t) \in
 \C^n\times \T^n_{\s}\times Y_{s}$, \\ $ x(t)= (r(t),\theta(t),\zeta(t)) $, 
which satisfy
\begin{align}\begin{split}\label{estim-phi1}
&|r(t) -r^0|\leq  4\eta^{-1}\big(1+\mu^{-1}\|\zeta^{0}\|_{s}+\mu^{-2}|r^0|+ \mu^{-2}||\zeta^0||^2_{s}\big)\fbp,\\
&|\theta(t)-\theta^0|\leq \mu^{-2}\fbp,\\
&\|\zeta(t)-\zeta^0\|_{s+\b} \leq 
\left( \mu^{-2}  \|\zeta^0\|_{s}+\mu^{-1}\right)\fbp,\\
\end{split}\end{align}
Moreover, the $\r$-derivative  of the mapping $x^0\mapsto x(t)$ 
 satisfies  the same estimates as  the increments $x(t)-x^0$. 
 \end{lemma}

\proof
Consider  the  equation for $\zeta(t)$ in \eqref{hameq-jet}:
 \begin{equation}\label{000}
 \dot \zeta(t)= a(t)+ B(t) \zeta(t),\quad \zeta(0)=\zeta^0 \in \O_{\mu-2\nu}(Y_s),
 \end{equation}
 where $a(t)=Jf_\zeta(\theta(t))$ is an analytic curve $[0,1]\to Y_\ga$ and 
 $B(t)=Jf_{\zeta\zeta}(\theta(t))$ is an analytic  curve $[0,1]\to \M$. 
  Both analytically depend on $\theta^0$. 
 By the hypotheses and using \eqref{Cauchy}
 \be\label{aB} ||a(t)||_s\leq \mu^{-1}\ff  ,
 \quad
 \| B\|_{\L(Y_s, Y_s)}\leq \mu^{-2}\ff\leq \frac 12 \nu\leq \frac 12 .\ee
  On the other hand by Lemma \ref{product} (iii), $B\in  \L(Y_s, Y_{s+\b})$ and
  \be\label{Bssb}
   \| B\|_{\L(Y_s, Y_{s+\b})}\leq \mu^{-2}\fbp.
  \ee
 By re-writing \eqref{000} in  the integral form 
 $\zeta(t)=\zeta^0 +\int_0^t (a(t')+B(t')\zeta(t'))\dd t'$ and iterating this relation, we get that 
\be\label{zeta}
 \zeta(t)=a^\infty(t)+(I+B^\infty(t))\zeta^0,
 \ee
 where
$$
 a^{\infty}(t)=\int_{0}^{t}a(t_{1})\text{d}t_{1}+\sum_{k\geq 2}\int_{0}^{t}\int_{0}^{t_{1}}\cdots \int_{0}^{t_{k-1}} \prod_{j=1}^{k-1}B(t_{j})a(t_{k})\text{d}t_{k}  \cdots \text{d}t_{2}\,\text{d}t_{1},
$$
 and
   $$
  B^{\infty}(t)=\sum_{k\geq 1}\int_{0}^{t}\int_{0}^{t_{1}}\cdots \int_{0}^{t_{k-1}} \prod_{j=1}^{k}B(t_{j})\text{d}t_{k}  \cdots \text{d}t_{2}\,\text{d}t_{1}.
$$
Due to   \eqref{hyp-f}, \eqref{aB} and \eqref{Bssb}, for each $k$ and for $0\le t_k\le\dots t_1\le1$ 
 we have that
$$
\|B(t_1)\dots B(t_k)\|_{\L(Y_s, Y_{s+\b})}\le (\frac 12)^{k-1}\mu^{-2} [f]^{s,\b+}_{\s,\mu,\D} .
$$
By this relation and  \eqref{aB} 
  we get that $a^\infty$ and $B^\infty$ are well defined for $t\in[0,1]$ 
and satisfy 
 \begin{align}
 \begin{split}\label{aB-inf}  
   \|B^\infty(t)\|_{\L(Y_s, Y_{s+\b})}&\leq  \mu^{-2}\fbp ,
 \\
 \|a^\infty(t)\|_{s+\b}& \leq  \mu^{-3} ( \fbp )^2 \leq \mu^{-1}\fbp   .
  \end{split}\end{align}
Again,  the curves  $a^\infty$ and $B^\infty$ analytically depend on $\theta^0$. 
 Inserting \eqref{aB-inf}  in \eqref{zeta}  we get that $\zeta=\zeta(t)$ satisfies the third estimate of
  \eqref{estim-phi1}.   \\
On the other hand  for all $t\in[0,1]$, $B\in \msb^+$ and 
$$|B(t)|_{s,\b+}\leq\mu^{-2}\fbp.$$
Therefore using Lemma \ref{product} we get
 \begin{align}
 \begin{split}\label{aB-b}  
   |B^\infty(t)|_{s,\b+}&\leq  \mu^{-2}\fbp .
  \end{split}\end{align}
  Since in \eqref{Phi1} $U( \theta;t)=I+B^\infty(t)$, then the estimates on $U$ in \eqref{4.55}  follow from \eqref{aB-inf} and \eqref{aB-b}. 
  \smallskip

 Now  consider  equation for $r(t)$:
 $$
 \dot r(t)= -\alpha(t)-\Lambda(t) r(t),\quad r(0)=r^{0}\in\O_{(\mu-2\nu)^2}(\C^n) 
 $$
 where $ \Lambda (t)= \nabla_\theta f_r(\theta(t))$ and
 \be \label{alpha}
 \alpha(t)= \nabla_\theta f_{\theta}(\theta(t))+\langle \nabla_\theta f_\zeta(\theta(t)),\zeta(t)\rangle+\frac12 \langle \nabla_\theta f_{\zeta\zeta}(\theta(t))\zeta(t),\zeta(t) \rangle.\ee
 The curve of matrices $\Lambda(t)$ and the curve of vectors $\alpha(t)$ analytically depend on $\theta^0\in \T^n_{\sigma-2\eta}$. Besides, $\alpha(t)$ analytically depends on $\zeta^0\in Y_s$, 
 while $\Lambda$ is  $\zeta^0$-independent.
 
 By the Cauchy estimate and \eqref{hyp-f}, for any $\theta(t)\in \T^n_{\sigma-\eta}$ we have 
 \be\begin{split}
 \label{l-la}
&|\Lambda(t)|_{\L(\C^n,\C^n)}\le \eta^{-1}\mu^{-2}\ff\le \frac 12,\\
&|\alpha(t)|\le 2\eta^{-1}  \ff (1+\mu^{-1}\|\zeta^0\|_{s}+\mu^{-2}\|\zeta^0\|^2_{s})
\end{split}
 \ee
 where for the second estimate we used that $\|\zeta(t)-\zeta^0\|_{s}\leq 1+\|\zeta^0\|_{s}$.\\
 Since $\nabla_{\zeta(t)}\alpha(t)=\nabla_\theta f_\zeta(\theta(t)) + 
 \nabla_\theta f_{\zeta\zeta}(\theta(t))\zeta(t)$ and
 $\nabla_{\zeta_0}={}^tU(\theta;t)\nabla_{\zeta}$,   then using \eqref{4.55} and Lemma \ref{product} we obtain
 \be\label{l-la1}
 \|\nabla_{\zeta^0}\alpha(t)\|_{s+\b}\le 4\eta^{-1}\mu^{-1}\fbp(1+\mu^{-1}\|\zeta^0\|_{s}).
 \ee
 Since 
 $\nabla^2_{\zeta^0}\alpha(t) = {}^tU
 \nabla^2_{\zeta(t)} \alpha(t) U={}^tU  \nabla_\theta   f_{\zeta\zeta}(\theta(t))U$,
 then  due to  \eqref{4.55} 
  and Lemma \ref{product}
  \be\label{l-la3}
 |\nabla^2_{\zeta^0}\alpha(t)|_{s,\b+}\le4
  \eta^{-1}\mu^{-2}\fbp .
 \ee
We proceed as for the $\zeta$-equation to derive 
\be\label{r}
r(t)=-\alpha^\infty(t)+(1-\Lambda^\infty(t))r^0,
 \ee
 where
\be\label{alpha-inf}
\alpha^{\infty}(t)=\int_{0}^{t}\alpha(t_{1}) \text{d}t_{1}+\sum_{k\geq 2}\int_{0}^{t}\int_{0}^{t_{1}}\cdots \int_{0}^{t_{k-1}} \prod_{j=1}^{k-1}\Lambda(t_{j})\alpha(t_{k})\text{d}t_{k}  \cdots \text{d}t_{2}\,\text{d}t_{1},
\ee
 and
\be\label{Lambda-inf}
  \Lambda^{\infty}(t)=\sum_{k\geq 1}\int_{0}^{t}\int_{0}^{t_{1}}\cdots \int_{0}^{t_{k-1}} \prod_{j=1}^{k}\Lambda(t_{j})\text{d}t_{k}  \cdots \text{d}t_{2}\,\text{d}t_{1}.
\ee
Using \eqref{l-la} we get that 
\begin{align*}
|\Lambda^\infty(t)| _{\L(\C^n\times\C^n)}&\leq  \frac 12,  \\
 |\alpha^\infty(t)|_{\C^n}&\leq  2\eta^{-1}\big( 1+\mu^{-1}\|\zeta^{0}\|_{s}+\mu^{-2} \|\zeta^0\|^2_{s}\big)\ff.
 \end{align*}
  Since in \eqref{Phi1} $S(\theta;1)=I- \Lambda^\infty(t)$, then the first estimate in \eqref{4.55} follows.
 Since $\Lambda^\infty(t)$ in \eqref{r} is $\zeta^0$-independent, then $L(\theta,\zeta;t)=-
 \alpha^\infty(t)$. This is a quadratic in $\zeta^0$ expression, and the estimates \eqref{4.56}
 follow from \eqref{l-la1}--\eqref{l-la3} and in view of the estimate for $\Lambda^\infty$  above. \\
  Finally using the estimates for $\Lambda^\infty$ and $\alpha^\infty$ we get from \eqref{r}
that $r=r(t)$ satisfies \eqref{estim-phi1}${}_1$, as \eqref{estim-phi1}${}_{2}$ directly comes from \eqref{hameq-jet} and \eqref{norm2}.  \endproof

Next we study how the flow maps $\Phi^t_f$ transform  functions from  $ \Tb$. 

 \begin{lemma}\label{composition}
  Let $0<2\eta<\s\leq 1$, $0<2\nu<\mu\leq 1$. Assume that  
 $f =f^T\in\Tbp$ satisfies \eqref{hyp-f}.
Let $h\in \Tb$ and denote for $0\leq t\leq 1$ 
 $$h_t(x;\r)=h(\Phi^t_f(x;\r);\r).$$ Then 
 $h_t\in \Tc^{s,\b}(\s-2\eta,\mu-2\nu,\D)$  and
$$
[h_t]^{s,\b}_{\s-2\eta,\mu-2\nu,\D} 
\leq C \frac{\mu}{\nu}\hhh
$$
where $C$ is an absolute constant.
\end{lemma}
\proof  
Let us write the flow map $\Phi^t_f$ as 
$$x^0=(r^0,\theta^0,\zeta^0)\mapsto x(t)=(r(t),\theta(t),\zeta(t)).$$
By Lemma~\ref{changevar}, $h_t(x^0)$ is analytic in $x^0\in \O(\s-2\eta,\mu-2\nu)$. Clearly $|h_t(x^0,\cdot)|\leq \hhh$ for $x^0\in \O(\s-2s,\mu-2\nu)$ and $\r\in\D$. So it remains to estimate the gradient and  hessian of $h(x^0)$.

1) {\it Estimating the  gradient.} Since $\theta(t)$ does not depend on $\zeta^0$,  we have
$$
\frac{\partial h_t }{\partial \zeta^0}= 
\sum_{k=1}^n \frac{\partial h(x(t))}{\partial r_k}\ \frac{\partial r_k(t)}{\partial\zeta^0}+ \sum_{b\in\L} \frac{\partial h(x(t))}{\partial \zeta_b(t)}\ \frac{\partial \zeta_b(t)}{\partial\zeta^0}=\Sigma_1+\Sigma_2.
$$
 i) Since $x(t)\in \O(\s-\eta,\mu-\nu)$,  
 we get  by the Cauchy estimate that
 $$
\left|  \frac{\partial h(x(t))}{\partial r_k}\right|\leq   \frac 1{3\nu^2}\hh .$$
As $\nabla_{\zeta^0} r_k(t)$ was estimated in \eqref{4.56}, then using \eqref{hyp-f} we get 
\begin{align*}
\|\Sigma_1\|_{s+\b} &\le C \nu^{-2} \hhh \,\eta^{-1}\mu^{-1}  \fb\le C      \mu^{-1}  \hhh\,.
\end{align*}

ii) Noting that $\Sigma_2(r,\theta,\zeta)={}^tU(\theta;t)\nabla_\zeta h$, we get using \eqref{4.55}:
$$\|\Sigma_2\|_{s+\b} \le 4  \mu^{-1} \hhh.$$
Estimating similarly $\frac{\partial}{\partial \r}\frac{\partial h_t}{\partial \zeta}$ we see that for
 $x\in \O(\s-2\eta,\mu-2\nu)$
$$
\| \partial_\r \nabla_{\zeta^0} h_t \|_{s+\b}\leq C\mu^{-1}\hhh.$$

2) {\it Estimating the hessian.}  Since $\theta(t)$ does not depend on $\zeta^0$ and since 
$\zeta(t)$ is affine in 
 $\zeta^0$, then 
 \begin{equation}
 \begin{split}
 \label{feo}
\frac{\partial^2 h_t}{\partial \zeta^0_a \partial \zeta^0_b}(x)&=
\frac{\partial^2 h(x(t))}{\partial \zeta\partial \zeta}  \frac{\partial \zeta(t)}{\partial \zeta^0_a}
\frac{\partial \zeta(t)}{\partial \zeta^0_b}
+
\frac{\partial^2 h(x(t))}{\partial  r^2}  
 \frac{\partial r(t)} {\partial \zeta^0_a}
\frac{\partial r(t)}{\partial \zeta^0_b} \\  
&+
\frac{\partial^2 h(x(t))}{\partial  r \partial\zeta}  
 \frac{\partial r(t)} {\partial \zeta^0_a}
\frac{\partial \zeta(t)}{\partial \zeta^0_b}
+
\frac{\partial h(x(t))}{\partial  r }  
 \frac{\partial^2 r(t)} {\partial \zeta^0_a  \partial \zeta^0_b}\\
& =: \Delta_1+\Delta_2+\Delta_3+\Delta_4.
\end{split}
\end{equation}

i) We have $|\partial^2 h/\partial \zeta_a\partial\zeta_b|_\b\le C\mu^{-2}\hhh$. 
Using this estimate jointly with \eqref{4.55} and Lemma \ref{product} we see that 
$$
|\Delta_1|_{\b}\le C\mu^{-2} \hhh.
$$

ii) Since for $x^0\in  \O^{s}(\s-2s,\mu-2\nu)$ by \eqref{4.56} we have 
$$
 \|\nabla_\zeta r_j\|_{s+\b}\le C\eta^{-1}\mu^{-1}\fbp\,,
$$
 and since by Cauchy estimate $|d_r^2h|\le C \nu^{-4}  \hhh$, we get using Lemma \ref{product}(v) and \eqref{hyp-f}
$$
 |\Delta_2|_{\b}\le C   \nu^{-4}\hhh   \eta^{-2}\mu^{-2}(\fbp)^2
 \le C\mu^{-2} \hhh\,.
$$

iii) For any $j$ we have by the Cauchy estimate that 
$\| \frac{\p}{\p r_j}\nabla_\zeta h\|_{s+\b}\le C \nu^{-3} \hhh$. Therefore by \eqref{4.55}
$$
\Big\|\sum_{a'} \frac{\partial^2 h}{\partial r_j \partial \zeta_{a'}} \frac{\partial\zeta_{a'}}{\partial \zeta^0_a}\Big\|_{s+\b}
\le C   \nu^{-3} \hhh\, .
$$
Since 
$$ 
\|\nabla_{\zeta^0}r_j\|_{s+\b} \le C\eta^{-1}\mu^{-1}\fb\le C\nu^2\mu^{-1}
$$
by \eqref{4.56}, then using  Lemma \ref{product} (v)  we find that $$
|\Delta_3|_{\b}\le C\nu^{-1}\mu^{-1} \hhh\,.
$$

iv) As $|\partial h/\partial r(x(t))|\le \nu^{-2}\hhh$ and 
$$
\left| \frac{\partial^2 r}{\partial\zeta^0_a \partial\zeta^0_b} \right|_{\b}
\le C\eta^{-1}\mu^{-2}\fbp
$$
by \eqref{4.56}, then 
$$
|\Delta_4|_{\b}\le C\mu^{-2}   \hhh.
$$

The $\rho$-gradient of the hessian leads to estimates similar to the above. So
the lemma is proven.

\endproof

We summarize the results of this section into a proposition.

\begin{proposition}\label{Summarize}
Let  $0<\s'<\s\leq 1,$ $0<\mu'<\mu\leq 1$.There exists an absolute constant $C\geq 1$ such that
\begin{itemize}
\item[(i)]
if $f =f^T\in\Tb$ and 
\be\label{hyp-f'}  
[f]^{s,\b}_{\s,\mu,\D}\leq\\
 \frac1{2} (\mu-\mu')^2 (\s-\s'),
\ee
 then
 for  all $0\le t\le1$,  the Hamiltonian
 flow map  $\Phi^t_f$ is a $\Cc^1$-map 
$$\O^{s}(\s',\mu')\times\D \to\O^{s}(\s,\mu);$$
real holomorphic and symplectic for any fixed  $\rho\in \D$.
Moreover,
$$||\Phi^t_f(x,\cdot)-x||_{s,\D}\le 
C\bigg( \frac1{\s-\s'}+\frac{1}{\mu^2}\bigg) [f]^{s,\b}_{\s,\mu,\D}$$
for any $x\in \O^{s}(\s',\mu')$. 
 
\item[(ii)] if $f =f^T\in\Tbp$ and 
\be\label{hyp-f''}  
[f]^{s,\b+}_{\s,\mu,\D}\leq\\
 \frac1{2} (\mu-\mu')^2 (\s-\s'),
\ee
 then
 for all $0\leq t\leq 1$ and all $h\in \Tc^{s,\b}(\s,\mu,\D)$,
the function
$h_t(x;\r)=h(\Phi^t_f(x,\r);\r)$ belongs to $\Tc^{s,\b}(\s',\mu',\D)$
and
$$[h_t]^{s,\b}_{\s',\mu',\D} 
\leq C \frac{\mu}{(\mu-\mu')  }
 [h]^{s,\b}_{\s,\mu,\D}. $$
\end{itemize}

\end{proposition}

\begin{proof} Take $\s'=\s-2s$ and $\mu'=\mu-2\nu$ and apply Lemmas \ref{changevar} and \ref{composition}.
\end{proof}

\section{Homological equation}\label{shomo}
Let us first recall the KAM strategy. 
Let $h_0$ be the normal form   Hamiltonian given by \eqref{h0}
$$
h_0(r,\zeta,\r)=\langle \om_0(\r), r\rangle +\frac 1 2\langle \zeta, A_0\zeta\rangle$$
satisfying  Hypotheses~A1-A3. Let
 $f$ be a perturbation and 
 $$f^T=f_\theta+\langle f_r, r\rangle+\langle f_\zeta,\zeta\rangle+\frac 1 2 \langle f_{\zeta\zeta}\zeta,\zeta \rangle$$ 
  be its jet  (see \eqref{jet}).
 If $f^T$ were zero,  then $\{\zeta=r=0\}$ would be an invariant 
  $n$-dimensional torus for the Hamiltonian
$h_0+f$. In general we only know that $f$ is small, say $f=\O (\eps)$, and thus $f^T=\O (\eps)$. In order  to 
decrease  the error term 
 we search for a hamiltonian jet $S=S^T=\O (\eps)$ such that its  time-one flow map
$\Phi_S=\Phi_S^1$ transforms the Hamiltonian $h_0+f$ to
$$
(h_0+f)\circ \Phi_S=h+f^+,
$$
where $h$ is a new normal form, $\eps$-close to $h_0$, and the new perturbation $f^+$ is such that its jet
is much smaller than $f^T$. More precisely, 
$$
h=h_0+\tilde h,\qquad
\tilde h=c(\r)+\langle \chi(\r),r\rangle+ \frac 1 2 \langle \zeta, B(\r)\zeta\rangle=\O(\eps), 
$$
and 
$\ 
\left(f^+\right)^T=\O (\eps^2).
$

As a consequence of the Hamiltonian structure we have (at least formally) that
$$(h_0+f)\circ \Phi_S= h_0+\{ h_0,S \}+f^T+ \O (\eps^2).$$
So to achieve the goal above 
we should  solve the {\it homological equation}:
\be \label{eq-homo}
\{ h_0,S \}=\tilde h-f^T+\O (\eps^2).
\ee
Repeating iteratively 
the same procedure with $h$ instead of $h_0$  etc., we will be forced to solve the homological equation, not
only for the normal form Hamiltonian \eqref{h0}, but for 
more general normal form Hamiltonians
\eqref{h} with $\om$ close to $\om_0$ and $A$ close to $A_0$ .

\medskip 

In this section we will consider a homological equation \eqref{eq-homo}  with $f$  in $\Tc^{s,\beta}(\s,\mu,\D)$ and we will build a solution $S$  in  $\Tc^{s,\beta+}(\s,\mu,\D)$. In this section, constants $C$ may take different values, but will only depend on $s$, $\beta$, $n$, $d^{*}$, $\gamma$, $c_{0}$, $\alpha_{1}$ and $\alpha_{2}$ given in Hypothesis A1, A2 and A3.

\subsection{Four components of the homological equation}\label{ss6.1}

Let $h$ be a  normal form Hamiltonian \eqref{h},
$$
h(r,\zeta,\r)=\langle \om(\r), r\rangle +\frac 1 2\langle \zeta, A(\r)\zeta\rangle\,,
$$
and let us write a jet-function  $S$ as 
$$
S(\theta,r,\zeta)=S_\theta(\theta)+\langle S_r(\theta), r\rangle+\langle S_\zeta(\theta),\zeta\rangle+
\frac 1 2 \langle S_{\zeta\zeta}(\theta)\zeta,\zeta \rangle.
$$ 
Therefore  the Poisson bracket of $h$ and $S$ 
 equals
\begin{equation*}
\begin{split}
\{h, S\}=
(\nabla_\theta\cdot\omega) S_\theta &+\langle (\nabla_\theta\cdot\omega) S_r,r\rangle +
\langle (\nabla_\theta\cdot\omega) S_\zeta, \zeta\rangle \\
&+
\frac12 \langle (\nabla_\theta\cdot\omega) S_{\zeta\zeta},\zeta\rangle 
-\langle AJ S_\zeta,\zeta\rangle  +\langle  S_{\zeta\zeta}JA\zeta,\zeta\rangle.
\end{split}
\end{equation*}
Accordingly the  homological equation \eqref{eq-homo} with $h_0$ replaced by $h$ 
 decomposes into four linear equations. The first two are
\begin{align}\label{homo1}
\langle \nabla_\theta	 S_\theta,\om\rangle =&-f_\theta+c+\O (\eps^2),\\ \label{homo2}
\langle \nabla_\theta	 S_r, \om\rangle =&-f_r+\chi+\O (\eps^2).
\end{align}
In these equations, we are forced to choose 
$$
c(\r)=\( f_\theta(\cdot,\r)\)  \quad \text{and}\quad \chi(\r) =\( f_r(\cdot,\r)\)$$
where $\(f\)$ denotes averaging of a function $f$ in $\theta\in\T^n$, to get that the space mean-value of the r.h.s. vanishes. 
The other two equations are 
\begin{align}\label{homo3}
 \langle\nabla_\theta	 S_\zeta, \om \rangle  -AJ
  S_\zeta=&- f_\zeta+\O (\eps^2),\\ \label{homo4}
 \langle\nabla_\theta	 S_{\zeta\zeta},\om\rangle   -
 AJS_{\zeta\zeta}  +
 S_{\zeta\zeta}JA=&-
 f_{\zeta \zeta}+B+\O (\eps^2)\,,
\end{align}
where the operator $B$ will be chosen later. 
The most delicate, involving the small divisors (see \eqref{D33}),
 is the last equation.

\subsection{The first two equations}\label{homogene}
We begin with equations 
\eqref{homo1} and \eqref{homo2} which are both of the form 
\be\label{homo0}
\langle\nabla_\theta \phi(\theta,\r),\om(\r)\rangle= \psi(\theta,\r)\ee
with $\(\psi\)=0$. 
Here $\om:\D\to\R^n$  is $\Ca^1$ and verifies 
$$|\om-\om_0|_{\Ca^1(\D)}\le\delta_0.$$
Expanding $\phi$ and $\psi$ in Fourier series,
$$\phi=\sum_{k\in\Z^n\setminus \{0\}}\hat \phi (k) e^{ik\cdot \theta}, \quad \psi=\sum_{k\in\Z^n\setminus \{0\}}\hat \psi (k) e^{ik\cdot \theta},$$
we solve eq.  \eqref{homo0}  by choosing
$$
\hat \phi (k) =\ - \frac{i}{ \langle\om, k\rangle}\hat\psi(k), \quad k\in\Z^n\setminus \{0\}; \qquad \hat\phi(0)=0.
$$
 Using Assumption A2 we have, for each $k\not=0$, either that
$$ |\langle\om(\r), k\rangle|\ge \delta_0\quad \forall \r$$
or that
$$\nazz  (\langle k,\omega(\r)\rangle)  \geq \delta_0\quad \forall \r
$$
for a suitable choice of a unit vector $\zz$.
The second case implies that 
$$
 |\langle\om(\r), k\rangle|\ge \ka\,,
 $$
 where $ \ka\le  \delta_0$,
for all $\r$ outside some open 
set $F_k\equiv F_k(\om)$ of Lebesgue measure $\leq \delta_0^{-1}\ka$.\\
Let
$$\D_1=\D\setminus \bigcup_{0<|k|\le N}F_k.$$
Then the closed set $\D_1$ satisfies 
$$\meas(\D\setminus \D_1)\leq N^n\frac {\ka}{\delta_0}\,,$$
and $ |\langle\om(\r), k\rangle|\ge  \ka$ for all $\r\in\D_1$.
Hence, for $\r\in\D_1$ and all $0<|k|\leq N$ we have 
 $$ |\hat \phi (k)| \le \frac{1}{\ka}|\hat\psi(k)|\,.
 $$
Setting $\phi(\theta,\r)=\sum_{0<|k|\leq N}\hat \phi (k,\r) e^{ik\cdot \theta}$, we get that
\be\label{homo0'}
 \langle\nabla_\theta \phi(\theta,\r), \om(\r) \rangle= \psi(\theta,\r)+R(\theta,\r).
\ee
Hence $\phi$ is an approximate solution of eq.~\eqref{homo0} with  the error term  
$R(\theta,\r) =-\sum_{|k|> N}\hat \psi (k,\r) e^{ik\cdot \theta}.$
We obtain by a classical argument  that for
 $(\theta,\r)\in \T^n_{\s'}\times \D_1$, $0<\s'<\s$, and $j=0,1$
 \be
\begin{split}\label{esti}
&| \phi(\theta,\r)|\leq \frac{C}{\ka^{}  {(\s-\s')^{n}}} \sup_{|\Im\theta|<\s}|\psi(\theta,\r)| ,\\
&|\p_\r^j  R(\theta,\r)|\leq \frac{C \,e^ {{-\frac12(\s-\s')N}}}{  {(\s-\s')^{n}}}\sup_{|\Im\theta|<\s} 
|\p_\r^j \psi(\theta,\r)| \,, 
\end{split}
\ee
where $C$ only depends on $n$. If $\psi$ is a real function, then so are
$\phi$ and $R$. \\
Differentiating in $\r$  the definition of 
$\hat\phi(k)$ gives \footnote{ Here and below 
  $\chi_Q(k)$ stands for the characteristic function of a set $Q\subset \Z^n$. }
 $$\p_\r \hat\phi(k)= \chi_{|k|\le N}(k)
 \Big(
 \ -\frac{i}{ \langle\om, k\rangle}\p_\r \hat\psi(k)+
 \ \frac{i}{ \langle\om, k\rangle^2} \langle \p_\r\om, k\rangle\hat\psi(k)\Big)\,.
 $$
 From this we derive that
\begin{align*}
|\p_\r \phi(\theta,\r)|\leq &\frac{C(|\omega_0(\rho)|_{C^1}+1)N}{\ka^{2}  { (\s-\s')^{n}}}\big( \sup_{|\Im\theta|<\s}|\psi(\theta,\r)|
+\sup_{|\Im\theta|<\s}|\p_\r \psi(\theta,\r)|\big)     \,   ,
\end{align*}
where we estimated the derivative of $\om$ by 
$|\omega_0(\rho)|_{C^1}+\delta_0\le |\omega_0(\rho)|_{C^1}+1$.

Applying this construction to  \eqref{homo1} and \eqref{homo2}  we get

\begin{proposition}\label{prop:homo12}
Let $\om:\D\to\R^n$  be  $\Ca^1$ and verifying 
${|\om-\om_0|_{\Ca^1(\D)}\le\delta_0.}$
Let $f\in \Tc^s(\s,\mu,\D)$ and let $\delta_0\geq \ka>0$, $N\ge1$.
Then there exists  a closed  set $\D_1= \D_1(\om,\ka,N)\subset \D$,  satisfying
$$\meas (\D\setminus \D_1)\leq  C N^{n} \frac\ka{\de_0}, $$
and
\begin{itemize}
\item[(i)] there exist real $\Ca^1$-functions $S_\theta$ and $R_\theta$ on $\T^n_{\s}\times \D_1\to\C$,  
analytic in $\theta$,
such that
 $$\langle \nabla_\theta	 S_\theta(\theta,\r),\om(\r)\rangle  =-f_\theta(\theta,\r)+\(f_\theta(\cdot,\r)\)+R_\theta(\theta,\r)$$
and for all $(\theta,\r)\in \T^n_{\s'}\times \D_1$, $\s'<\s$, and $j=0,1$
 \begin{align*}
 |\p_\r^jS_\theta(\theta,\r)|\leq &\frac{CN}{\ka^2 
  {(\s-\s')^{n}}}[f]^s_{\s,\mu,\D_1},\\
 |\p_\r^j R_\theta(\theta,\r)|\leq & \frac{C  { e^{- \frac12(\s-\s')N}}}  {  {(\s-\s')^{n}}}  [f]^s_{\s,\mu,\D_1}\,.
\end{align*}

\item[(ii)] there exist  real $\Ca^1$ vector-functions $S_r$ and  $R_r$  on $\T^n_{\s}\times \D_1$,  
analytic in $\theta$, such that
 $$\langle \nabla_\theta	 S_r(\theta,\r),\om(\r)\rangle =-f_r(\theta,\r)+\( f_r(\cdot;\r)\)+R_r(\theta,\r),$$
 and for all $(\theta,\r)\in \T^n_{\s'}\times \D_1$, $\s'<\s$, and $j=0,1$
 \begin{align*}
 |\p_\r^jS_r(\theta,\r)|\leq &
 \frac{C}{\ka^2
  { (\s-\s')^{n}}}[f]^s_{\s,\mu,\D_1},\\
 |\p_\r^j R_r(\theta,\r)|\leq & \, \frac{C  { e^{- \frac12(\s-\s')N}}}  {  {(\s-\s')^{n}}}  [f]^s_{\s,\mu,\D_1}
 \,.
\end{align*}
\end{itemize}
The  constant $C$ only depends on $ |\om_0 |_{\Ca^1(\D)}$.
\end{proposition}


\subsection{The third equation}\label{s5.3} 
To begin with, we  recall a result proved
in the appendix of \cite{EK10}.

\begin{lemma}\label{EK}
Let $A(t)$ be a real diagonal $N\times N$-matrix with diagonal components $a_j$ which are $\Ca^1$ on $I=]-1,1[$,
 satisfying for all $j=1,\dots,N$ and all $t\in I$
$$a'_j(t)\geq\delta_0
 .$$ Let $B(t)$ be a Hermitian $N\times N$-matrix of class $\Ca^1$ on $I$ such that \footnote{Here 
 $\|\cdot\|$ means the operator-norm of a matrix associated to the euclidean norm on $\C^N$.}
$$\|B'(t)\|\leq \delta_0/2 
,$$ 
for all $t\in I$. 
Then
$$\|(A(t)+B(t))^{-1}\|\leq \frac 1 \eps $$
outside a set of $t\in I$ of Lebesgue measure $\leq C N\eps\delta_0^{-1}$,
where 
$C$ is a numerical constant.
\end{lemma}

Concerning the third component  \eqref{homo3} of the homological equation we have

\begin{proposition}\label{prop:homo3}
Let $\om:\D\to\R^n$  be  $\Ca^1$ and verifying 
${|\om-\om_0|_{\Ca^1(\D)}\le\delta_0.}$  Let $\D\ni\r\mapsto A(\r)\in \NF\cap\M_0$ be $\Ca^1$ and verifying 
\be\label{ass1}
 \| \p_\r^j (A(\r)-A_0)_{[a]} \|\le \frac{1}{2}\de_0
 \ee
for $j=0,1$, $a\in\L$ and  $\r\in \D$.
Let $f\in \Tc^s(\s,\mu,\D)$, $0<\ka\leq \min(\frac{\delta_0}2, \frac{c_{0}}{2})$ and  $N\ge 1$. \\
Then there exists  a closed set $\D_2=\D_2(\om,A,\ka,N)\subset \D$,  satisfying
 $$
 \meas (\D\setminus \D_2)\leq  
 C N^{\exp} 
 \frac{\ka}{\de_0}, $$
and there exist real $\Ca^1$-functions $S_\zeta$ and $R_\zeta$ from  $\T^n  \times \D_2$ to $Y_s$,  
analytic in $\theta$, such that
\be\label{homo3ter}\langle  \nabla_\theta	 S_\zeta(\theta,\r), \om(\r)\rangle  { -A(\r) J}S_\zeta(\theta,\r)=
-f_\zeta(\theta,\r)+R_\zeta(\theta,\r)
\ee
and for all $(\theta,\r)\in \T^n_{\s'}\times \D_2$, $\s'<\s$, and $j=0,1$
 \begin{align*}
\mu \| \p_\r^j S_\zeta(\theta,\r)\|_{s+1}\leq &\frac{CN
}{\ka^2 (\s-\s')^{2n}}[f]^{s,\b}_{\s,\mu,\D},\\
\mu \|  \p_\r^j R_\zeta(\theta,\r)\|_{s}\leq &\frac{C e^{-\frac12(\s-\s')N}}{(\s-\s')^{n}}[f]^{s,\b}_{\s,\mu,\D}\ .
\end{align*}
The exponent $\exp$ only depends on $d^*, n,\ga$ while the constant $C$ also depends on $ |\om_0 |_{\Ca^1(\D)}$. 
\end{proposition}

\proof

It is more convenient to deal with the hamiltonian operator $JA$ than with operator $AJ$. Therefore we
multiply eq.~\eqref{homo3ter} by $J$ and obtain for $JS_\zeta$ the equation
\be\label{homo31}
\big\langle  \nabla_\theta	(J S_\zeta)(\theta,\r),\, \om(\r) \big\rangle   -JA(\r) (J S_\zeta)(\theta,\r)=
- Jf_\zeta(\theta,\r)+ JR_\zeta(\theta,\r)
\ee

Let us re-write \eqref{homo31} in the complex variables ${}^t(\xi,\eta)$. 
 For $a\in\L$ {
\be\label{U}
\zeta_a=\left( \begin{array}{ll}p_a\\ q_a\end{array}\right)=U_a\left( \begin{array}{cc}\xi_a\\ \eta_a\end{array}\right), \quad 
U_a=\frac 1 {\sqrt 2}\left( \begin{array}{cc}  1 &  1 \\ -i & i\end{array}\right)\,.
\ee
The symplectic operator $U_a$ transforms the quadratic form 
$(\lambda_a/2) \langle\zeta_a, \zeta_a\rangle$ to $i\lambda_a\xi_a\eta_a$. Therefore, if we denote by $U$
the direct product of the operators diag$\, (U_a, a\in\L)$  then
it transforms $(1/2)\langle\zeta, A_0\zeta\rangle\ $ to $\ \sum_{a\in\L}i\lambda_a \xi_a\eta_a$. 
So  it transforms the hamiltonian matrix 
$JA_{0}$
to the diagonal hamiltonian 
matrix
$$
\text{diag}\, \{ i \lambda_a \left( \begin{array}{cc}  -1 &  0 \\ 0 & 1\end{array}\right), a\in\L\}.
$$
Then we make { in \eqref{homo31} the 
substitution   $ JS_\zeta=U S$, $JR_\zeta=U R$ 
and $- Jf_\zeta=U F_\zeta$, where $S={}^t(S^\xi, S^\eta)$, etc. }
In this notation eq.~\eqref{homo3ter}  decouples into two equations 
\begin{align}\begin{split}\label{homo3bis}
&\langle \nabla_\theta  S^\xi, \om\rangle - i\ {}^tQS^\xi= F^\xi+R^\xi,\\
&\langle\nabla_\theta  S^\eta,\om \rangle+  iQS^\eta= F^\eta+R^\eta.
\end{split}\end{align}
Here $Q:\L\times\L\to \C$ is the scalar valued matrix associated to $A$ via the formula \eqref{Q}, i.e. $$Q=\diag \{ \la,\ a\in\L\}+B,
$$
where  $B$ is Hermitian and  block diagonal.

Written in the  Fourier variables,  eq.~\eqref{homo3bis} becomes 
\begin{align}\begin{split}\label{homo3-4}
i(\langle k, \om\rangle - \ {}^tQ)\ \hat S^\xi(k)&= \hat F^\xi(k)+\hat R^\xi(k),\quad k\in \Z^n,\\
i(\langle k, \om\rangle + Q)\ \hat S^\eta(k)&= \hat F^\eta(k)+\hat R^\eta(k),\quad k\in \Z^n.
\end{split}\end{align}
The two equations in \eqref{homo3-4} are similar, so let us consider (for example) the second one, and
let us  decompose it into  its ``components'' over the  blocks $[a]$:
\be\label{homo3-4bis}
i( \langle k, \om(\r) \rangle  + Q(\r)_{[a]}) \hat S_{[a]}(k)=\hat F_{[a]}(k,\r)+\hat R_{[a]}(k)\ee
where the matrix $Q_{[a]}$ is  the restriction of $Q$  to $[a]\times[a]$ and the vector $\hat F_{[a]}(k,\r)$ is the
restriction of $\hat F(k,\r)$ to $[a]$  --  here we have suppressed the upper index $\eta$.
Denoting by $L(k,[a],\r)$ the Hermitian operator in the left hand side of equation \eqref{homo3-4bis}, we want to estimate the operator norm of $L(k,[a],\r)^{-1}$, i.e. we look for a lower bound of the modulus of the eigenvalues of $L(k,[a],\r)$.

Let $\alpha(\r)$ denote an eigenvalue of the matrix $Q_{[a]}(\r)$, $a\in\L$.
It follows from \eqref{ass1} that 
$$|\alpha(\r) -\la |\leq  \frac{\delta_0} {2}\leq  \frac{c_0} {2}$$
for some appropriate $a\in[a]$, which implies that  
$$|\alpha(\r)|\ge\frac{c_0}2 w_a^{\ga}\ge 2 \ka w_a$$ 
by \eqref{laequiv}. Hence, 
$$
\|L(0,[a],\r)^{-1}\| \leq \big(\ka  w_a\big)^{-1}\quad \forall\r,\ \forall a.
$$
Assume that  $0<|k|\le N$.  Since $|\langle k, \om(\r)\rangle |\leq CN$ it follows from \eqref{laequiv}
 that
$$
|\langle k, \om(\r)\rangle  + \alpha(\r)|\geq \frac {c_{0}} 4 w_a^{\ga}\ge  \ka\ w_a$$
whenever $ w_a\geq( \frac{4CN}{c_0})^{\frac1{\ga}}$. Hence  for these $a$'s we get
\be\label{inverse}
\|L(k,[a],\r)^{-1}\| \leq \big(\ka \ w_a\big)^{-1}\quad \forall\r.
\ee
Now  let $ w_a\leq (\frac{4CN}{c_0})^{\frac1{\ga}}$. By Hypothesis A2 we have either
$$|\langle k, \om(\r)\rangle  + \la|\geq \delta_0 w_a\quad \forall\r, \forall a$$ 
or we have a unit vector ${\mathfrak z}$ such that
$$ \nazz(  \langle k,\om(\r)\rangle+\la) \geq \delta_0\quad \forall\r, \forall a.$$
 The first case clearly implies  \eqref{inverse},
so let us consider the second case.
By  \eqref{ass1} it follows that
$$ \| \nazz H_{[a]}(\r)   \| \le \frac{\de_0}2.$$
The Hermitian matrix $( \langle k, \om(\r) \rangle  + Q(\r)_{[a]})$ is of dimension $\lsim w_a^{d^*} $ (see \eqref{block})    therefore, by   Lemma \ref{EK}, we conclude that   \eqref{inverse} holds
for all $\r$ outside   a suitable set $F_{a,k}$ of measure 
$\lsim w_a^{d^*+1} \ka\delta_0^{-1}$ . Let
$$
\D_2=\D\setminus  \bigcup_{\substack{|k|\leq N\\  w_a \leq (\frac{4CN}{c_0})^{\frac{1}{\ga}}}}F_{a,k}.
$$
Then we get
$$
\meas(\D\setminus \D_2)\leq 
C N^n\Big (\frac{N}{c_0} \Big)^{\frac{d^*+2}{\ga}}\frac{\ka}{\delta_0} 
$$ 
and \eqref{inverse} holds for all $\r\in\D_2$, all $|k|\le N$ and all $[a]$.

Equation \eqref{homo3-4bis} is now  solved by
\be\label{homeq3}\hat S_{[a]}(k,\r)= \chi_{|k|\le N} (k)
L(k,[a],\r)^{-1}\hat F_{[a]}(k,\r), \quad  a\in\L\,,
 \ee
and 
\be\label{homeq32}
\hat R_{[a]}(k,\r)= { \chi_{|k| > N} (k)}
\hat F_{[a]}(k,\r ), \quad  a\in\L \,.
\ee
Using \eqref{inverse},  we have for $\r\in\D_2$ 
\begin{align*}
\|S_{[a]}(\theta,\r)\|\leq &\frac{C}{\ka\ w_a (\s-\s')^{n}}\sup_{|\Im\theta|<\s}\| F_{[a]}(\theta,\r)\|,\\
\| R_{[a]}(\theta,\r)\|\leq &\frac{Ce^{- \frac12(\s-\s')N}}{(\s-\s')^{n}}\sup_{|\Im\theta|<\s}| F_{[a]}(\theta,\r)|.
\end{align*}
for $\theta\in \T^d_{\s'}\,$, see \eqref{esti}. \\
Since $\|S\|^2_s=\sum_{a\in\L} w_a^{2s}|S_a|^2=\sum_{a\in\hat\L} w_a^{2s}\|S_{[a]}\|^2$ these estimates imply that 
\begin{align*}
\| S(\theta,\r)\|_{s+1}\leq &\frac{C}{\ka (\s-\s')^{n}}\sup_{|\Im\theta|<\s}\| F(\theta,\r)\|_{s},\\
\| R(\theta,\r)\|_{s}\leq &\frac{Ce^{- \frac12(\s-\s')N}}{ (\s-\s')^{n}}\sup_{|\Im\theta|<\s}\| F(\theta,\r)\|_{s},
\end{align*}
for any $\s'\le \s$. 
The estimates of the derivatives with respect to $\r$ are obtained by differentiating 
\eqref{homo3-4bis} to obtain
$$L(k,[a],\r)[\partial_\r\hat S_{[a]}(k)]=-[\partial_\r L(k,[a],\r)]\hat S_{[a]}(k)+[\partial_\r\hat F_{[a]}(k,\r)]+[\partial_\r\hat R_{[a]}(k)]
$$
which is an equation  of the same type as \eqref{homo3-4bis} for $\partial_\r\hat S_{[a]}(k)$ and $\partial_\r\hat R_{[a]}(k)$ where $-[\partial_\r L(k,[a],\r)]\hat S_{[a]}(k)+[\partial_\r\hat F_{[a]}(k,\r)]:=B_{[a]}(k,\r)$ plays the role of $\hat F_{[a]}(k,\r)$. We solve this equation as in \eqref{homeq3}-\eqref{homeq32} and we note that
$${ \chi_{|k| > N} (k)}
 B_{[a]}(k,\r)={ \chi_{|k| > N} (k)} [\partial_\r\hat F_{[a]}(k,\r)]
$$ and thus
$$\| R(\theta,\r)\|_{s}\leq \frac{Ce^{- \frac12(\s-\s')N}}{ (\s-\s')^{n}}\sup_{|\Im\theta|<\s}\| F(\theta,\r)\|_{s}.$$
On the other hand
$$\| B_{[a]}(k,\r)\|_s\leq 
 \frac{CN}{\ka (\s-\s')^{n}}\sup_{|\Im\theta|<\s}\| \ F(\theta,\r)\|_{s}+\sup_{|\Im\theta|<\s}\|\partial_\r F(\theta,\r)\|_{s}$$
 and therefore we get
 \begin{align*}
\| \partial_\r S(\theta,\r)\|_{s+1}\leq \frac{CN\mu^{-1}}{\ka^2 (\s-\s')^{2n}}[f]^{s,\b}_{\s,\mu,\D}.
\end{align*}

 The  functions $F$ and $R$ are complex,  and the constructed solution
 $S_\zeta$ may also be complex. Instead of proving that it is real, we replace $S_\zeta, \theta\in\T^n$,
 by its real part and then analytically extend it to $\T^n_{\sigma'}$, using the relation
 $
 \Re S_\zeta(\theta,\rho) :=\frac12 ( S_\zeta(\theta,\rho) +\bar  S_\zeta(\bar\theta,\rho) ).
 $
 Thus we obtain a real solution
 which obeys the same estimates.
\endproof

\subsection{The last equation}\label{s5.4}
 Concerning the fourth component of the homological equation, \eqref{homo4}, we have the following result
\begin{proposition}\label{prop:homo4}
Let $\om:\D\to\R^n$  be  $\Ca^1$ and verifying 
${|\om-\om_0|_{\Ca^1(\D)}\le\delta_0.}$  Let $\D\ni\r\mapsto A(\r)\in \NF\cap\M_{s,\b}$ be $\Ca^1$ and verifying 
\be\label{ass2}
 \left| \p_\r^j (A(\r)-A_0) \right|_{s,\b} \le \frac{\de_0}{4}\ee
for $j=0,1$ and $\r\in \D$.
Let $f\in \Tc^{s,\b}(\s,\mu,\D)$, $0<\ka\le\frac{\delta_0}2$ and $N\ge 1$. \\
 Then there exists a subset $\D_3=\D_3(h, \ka,N)\subset \D$, satisfying 
 $$\meas (\D\setminus \D_3)\leq  
 C \Big(\frac{ N}{c_0} \Big)^{\exp} \Big(\frac{\ka}{\de_0} \Big)^{\exp'}, $$
and there exist real $\Ca^1$-functions 
$B: \D_3\to \msb \cap \NF $, $ S_{\zeta\zeta}(\cdot;\r): \D_3\to \msb^+$ and 
 $R_{\zeta\zeta}(\cdot;\r):\T^n_{\s}\times \D_3\to \msb$,
analytic in $\theta$, such that
\be\label{homo4ter} 
\langle  \nabla_\theta S_{\zeta\zeta}(\theta,\r), \om(\r)\rangle  -A(\r)JS_{\zeta\zeta}(\theta,\r)+
S_{\zeta\zeta}(\theta,\r)JA(\r)=-f_{\zeta \zeta} (\theta,\r)+B(\r)+R_{\zeta\zeta}(\theta,\r)
 \ee
and for all $(\theta,\r)\in \T^n_{\s'}\times \D_3$, $\s'<\s$, and $j=0,1$
  \begin{align}
\label{estim-homo4R}
\mu^2\left| \p_\r^j R_{\zeta\zeta}(\theta,\r)\right|_{s,\b}&\leq  C\ \frac{e^{-\frac12(\s-\s')N}}{ (\s-\s')^{n}}[f]^{s,\b}_{\s,\mu,\D},\\
\label{estim-homo4S}
\mu^2 \left|\p_\r^j  S_{\zeta\zeta}(\theta,\r)\right|_{s,\b+}&\leq C\
\frac{N^{1+d^*/\gamma}}{\kappa^{2+d^*/2\beta}(\s-\s')^{n}}[f]^{s,\b}_{\s,\mu,\D},\\
 \label{B}
\mu^2 \left|\p_\r^j B(\r)\right|_{s,\b}&\leq  C\   [f]^{s,\b}_{\s,\mu,\D}.\end{align}
The two exponents $\exp$ and $\exp'$ are positive numbers depending on $n$, $\ga$, $d^*$, $\a_1$, $\a_2$, $\b$. The constant $C$ also depends on $|\om_{0}|_{C^{1}(\D)}$.
\end{proposition}
\proof
As in the previous section, and using the same notation, we re-write \eqref{homo4ter} in complex variables. 
So we introduce 
$S={}^t\! U S_{\zeta,\zeta} U$, $R={}^t\! U R_{\zeta\zeta} U$ and $F={}^t\! U f_{\zeta\zeta} U$.
\\
 By construction, $ S^b_a\in \M_{2\times 2}$ for all $a,b\in\L$. Let us denote  $$ S^b_a= \left( \begin{array}{cc} (S^b_a)^{\xi\xi} & (S^b_a)^{\xi\eta} \\   (S^b_a)^{\xi\eta} & (S^b_a)^{\eta\eta}\end{array}\right) $$ and then $$S^{\xi\xi}=((S^b_a)^{\xi\xi})_{a,b\in\L}, \quad S^{\xi\eta}=((S^b_a)^{\xi\eta})_{a,b\in\L},\quad S^{\eta\eta}=((S^b_a)^{\eta\eta})_{a,b\in\L}.$$ We use similar notations for $R$, $B$ and $F$.\\
In this notation  \eqref{homo4ter}  decouples into three equations
\footnote{Actually \eqref{homo4ter}   decomposes into four scalar equations but the fourth one is the transpose of the third one.}
\begin{align*}
&\langle \nabla_\theta  S^{\xi\xi},\om \rangle +i Q S^{\xi\xi}+ i S^{\xi\xi}\ {}^tQ= B^{\xi\xi}- F^{\xi\xi}+ R^{\xi\xi},\\ 
&\langle \nabla_\theta  S^{\eta\eta}, \om\rangle  -i\ {}^tQ S^{\eta\eta} -i  S^{\eta\eta} Q= B^{\eta\eta}- F^{\eta\eta}+ R^{\eta\eta},\\ 
&\langle \nabla_\theta  S^{\xi\eta},\om\rangle  + iQ S^{\xi\eta} -  iS^{\xi\eta} Q= B^{\xi\eta}- F^{\xi\eta}+ R^{\xi\eta}\,,\end{align*}
where we recall that $Q$ is the scalar valued matrix associated to $A$ via the formula \eqref{Q}. {
The first and the second equations are of the same type, so we focus on the resolution of   the second and the 
third equations.}
Written in Fourier variables,  they read
 \begin{align} 
\label{homo4-2}
&i(\langle k ,\om\rangle -{}^tQ)\hat S^{\eta\eta}(k) - i\hat S^{\eta\eta}(k) Q= \delta_{k,0}B^{\eta\eta}-\hat F^{\eta\eta}(k)+\hat R^{\eta\eta}(k),\quad k\in \Z^n,\\ \label{homo4-3}
&i(\langle  k,\om\rangle +Q)\hat S^{\xi\eta}(k) - i\hat S^{\xi\eta}(k) Q= \delta_{k,0}B^{\xi\eta}-\hat F^{\xi\eta}(k)+\hat R^{\xi\eta}(k),\quad k\in \Z^n   \,,
\end{align}
where  $\delta_{k,j}$ denotes the Kronecker symbol. 

\medskip

{\it Equation \eqref{homo4-2}. } We  { chose $B^{\eta\eta}=0$ and}
 decompose the equation into ``components'' on each product block
$[a]\times[b]$:
\be\label{homo+}
L\, \hat S_{[a]}^{[b]}(k) 
= i\hat F_{[a]}^{[b]}(k,\r){ -i}\hat R_{[a]}^{[b]}(k)
\ee
where we have suppressed the upper index ${\eta\eta}$ and the operator $L:= L(k,{[a]},{[b]},\r)$ is the linear   Hermitian operator, acting in the space of complex
$[a]\times[b]$-matrices defined by
$$
L\, M= \big(
\langle k, \om(\r)\rangle  - {}^tQ_{[a]}(\r)\big) M 
- M Q_{[b]}(\r).$$
 The matrix $Q_{[a]}$ can be diagonalized in an orthonormal basis:
 $${}^tP_{[a]}Q_{[a]}P_{[a]}=D_{[a]}.$$
Therefore denoting $\hat{S'}_{[a]}^{[b]}={}^tP_{[a]}S_{[a]}^{[b]}P_{[b]}$, $\hat{F'}_{[a]}^{[b]}={}^tP_{[a]}F_{[a]}^{[b]}P_{[b]}$ and $\hat{R'}_{[a]}^{[b]}={}^tP_{[a]}R_{[a]}^{[b]}P_{[b]}$ the homological equation \eqref{homo+} reads
\be\label{homo++}( \langle k,\om\rangle +D_{[a]})\hat{S'}_{[a]}^{[b]}(k)-{S'}_{[a]}^{[b]}(k)D_{[b]}=i\hat{F'}_{[a]}^{[b]}(k){ -i}\hat {R'}_{[a]}^{[b]}(k).\ee
This equation can be solved term by term:
\be\label{R'}\hat{R'}_{j\ell}(k)= \hat{F'}_{j\ell}(k),\quad j\in[a],\ \ell\in[b],\ |k|> N\ee
and
\be\label{S'}\hat{S'}_{j\ell}(k)=\frac i{\langle k,\om(\r)\rangle \ -\alpha_j(\r)-\beta_\ell(\r)}\hat{F'}_{j\ell}(k),\quad j\in[a],\ \ell\in[b],\ |k|\leq N\ee
where $\alpha_j(\r)$ and $\beta_\ell(\r)$ denote eigenvalues of $Q_{[a]}(\r)$ and $Q_{[b]}(\r)$, respectively. First notice that by \eqref{R'} one has
 $$|R(\theta)|_{s,\b}=|R'(\theta)|_{s,\b}\leq \frac{Ce^{-\frac12 (\s-\s')N}}{ (\s-\s')^{n}}
\sup_{|\Im\theta|<\s}| F(\theta)|_{s,\b}.$$
To estimate $S$ we want to use Lemma \ref{delort} below.
As $Q_{[a]}=\diag\{\la\ : a\in[a]\}+B_{[a]}$ with $B$ Hermitian,  using  hypothesis \eqref{ass2}
 we get that 
\be\label{alpha-a}|(\alpha_j(\r)+\beta_\ell(\r)) -(\la+\lb)|\leq
 \left(\frac{\de_0}{4}+ \frac{\de_0}{4}\right)\frac 1{(w_{a}w_{b})^{\b}}\le \frac{\de_0} {2(w_{a}w_{b})^{\b}}.
 \ee
 Moreover, in order to apply Lemma \ref{delort} we have to estimate $|\alpha_{j}(\r)- \lambda_{a}|$ and $|\beta_{l}(r)-\lambda_{b}|$, this is done thanks to assumption \eqref{ass2} : 
 $$ |\alpha_{j} - \lambda_{a}| \leq \| Q_{[a]}(\r) - \lambda_{[a]} \mathrm{Id} \| \leq \frac{1}{w_{[a]}^{2\beta}} |A(\r) -A_{0}|_{s,\b}\leq \frac{\delta_{0}}{4 w_{[a]}^{2\b}}\,$$
 and the corresponding estimate holds for $|\beta_{l}(r)-\lambda_{b}|$.
 
It follows as in the proof of Proposition \ref{prop:homo3}, using Lemma~\ref{EK}, relation
\eqref{laequiv}, Assumption A2
and \eqref{ass2},  that there exists a subset 
$\D_2=\D_2(h,\ka,N)\subset \D$,  satisfying
 $$\meas (\D\setminus \D_2)\leq  C\Big(\frac{ N}{c_0}\Big)^{\exp} \frac{\ka}{\de_0}, $$
such that 
$$| \langle k,\om(\r)\rangle \ -\alpha_j(\r)-\beta_\ell(\r)|\geq {\ka}(1+|w_a+w_b|),$$
holds for all $\r\in\D_2$, all $|k|\le N$, all $j\in[a],\ \ell\in[b]$ and all $[a],[b]\in\hat\L$.
Thus for $\r\in\D_2$ we obtain by Lemma \ref{delort} that $\hat S'(k)\in\msb^+$ for all $|k|\leq N$ and
$$|\hat S'(k)|_{s,\b+}\leq C\ka^{-1-d^{*}/(4\beta)}N^{d^{*}/(2\gamma)}|\hat F'(k)|_{s,\b}.$$
Therefore we obtain a solution  $S$ satisfying for any $|\Im\theta|<\s'$
\begin{align*}
| S(\theta)|_{s,\b+}\leq &\frac{CN^{d^{*}/(2\gamma)}}{\ka^{1+d^{*}/(4\beta)} (\s-\s')^{n}}
\sup_{|\Im\theta|<\s}| F(\theta)|_{s,\b}.\\
\end{align*}

The estimates for the derivatives with respect to $\r$ are obtained by differentiating \eqref{homo+} which leads to (here we drop all the indices to simply the formula)
$$L(\partial_\r \hat S_{[a]}^{[b]}(k,\r))=-(\partial_\r L)\hat S_{[a]}^{[b]}(k,\r)+i\partial_\r \hat F_{[a]}^{[b]}(k,\r)-i\partial_\r \hat R_{[a]}^{[b]}(k,\r)\,,$$
which is an equation of the same type as \eqref{homo+} for $\partial_\r \hat S_{[a]}^{[b]}(k,\r)$ and $\partial_\r \hat R_{[a]}^{[b]}(k,\r)$ where $i\hat F_{[a]}^{[b]}(k,\r)$ is replaced by $B_{[a]}^{[b]}(k,\r)=-(\partial_\r L)\hat S_{[a]}^{[b]}(k,\r)+i\partial_\rho\hat F_{[a]}^{[b]}(k,\r)$. This equation is solved by defining
\begin{align*}\partial_\r \hat S_{[a]}^{[b]}(k,\r)= &\chi_{|k|\le N} (k)
L(k,[a],[b],\r)^{-1}B_{[a]}^{[b]}(k,\r),\\
 \partial_\r \hat R_{[a]}^{[b]}(k,\r)=&- i{ \chi_{|k| > N} (k)}
B_{[a]}^{[b]}(k,\r)={ \chi_{|k| > N} (k)}
\partial_\rho\hat F_{[a]}^{[b]}(k,\r)\,.
\end{align*}
Since
$$ |(\partial_\r L)\hat S(k,\r)|_{s,\b}\leq C(N (|\partial_\r \om_0|+\de_0)+2\delta_0 ) |\hat S(k,\r)|_{s,\b}\leq CN|\hat S(k,\r)|_{s,\b}\,,$$
we obtain
$$
|B(k,\r)|_{s,\b}\leq CN\ka^{-1-d^{*}/(4\beta}N^{d^{*}/2\gamma}|\hat F(k)|_{s,\b}
$$
and thus following the same strategy as in the resolution of \eqref{homo+} we get
\begin{align*}
\mu^{2}| \partial_\r S(\theta)|_{s,\b+}\leq &\frac{CN^{1+d^{*}/\gamma}}{\ka^{2+d^{*}/(2\beta)} (\s-\s')^{n}}
[f]^{s,\b}_{\s,\mu,\D},\\
\mu^{2}|\partial_\r R(\theta)|_{s,\b}\leq & \frac{Ce^{-\frac12 (\s-\s')N}}{ (\s-\s')^{n}}
[f]^{s,\b}_{\s,\mu,\D}.
\end{align*}

\medskip

{\it Equation \eqref{homo4-3}. } 
It remains to consider  \eqref{homo4-3} which decomposes into the ``components''
over the product blocks $[a]\times[b]$ (we have suppressed the upper index ${\xi\eta}$):
\be\label{homo-}
\begin{split}
 \langle k, \om(\r)\rangle \ 
\hat S_{[a]}^{[b]}(k) &+Q_{[a]}(\r)\hat S_{[a]}^{[b]}(k) - \hat S_{[a]}^{[b]}(k) Q_{[b]}(\r) \\&= 
-i\delta_{k,0}B_{[a]}^{[b]}+i \hat F_{[a]}^{[b]}(k,\r)-i\hat R_{[a]}^{[b]}(k).
\end{split}
\ee
 First we solve this case when $k=0$ and $w_a=w_b$  by defining $$\hat S_{[a]}^{[a]}(0)=0,\quad \hat R_{[a]}^{[a]}(0)=0\text{ and }B_{[a]}^{[a]}=\hat F_{[a]}^{[a]}(0).
$$
Then we impose $B_{[a]}^{[b]}=0$ for $w_a\neq w_b$ in such a way $B\in\msb\cap\NF$ and satisfies
$$
|B|_{s,\b}\leq | \hat F(0)|_{s,\b}.$$
The estimates of the derivatives with respect to $\r$ are obtained by differentiating the expressions for
$B$.

\medskip

Then, when $k\neq 0$ or $w_a\neq w_b$, with the same definition of $S'$, $F'$ as in \eqref{homo++} we obtain 
\be\label{homo+-}( \langle k,\om\rangle +D_{[a]})\hat{S'}_{[a]}^{[b]}(k)-{S'}_{[a]}^{[b]}(k)D_{[b]}=i\hat{F'}_{[a]}^{[b]}(k){ -i}\hat {R'}_{[a]}^{[b]}(k).\ee 
This equation can be solved term by term:
\be\label{R''}\hat{R'}_{j\ell}(k)= \hat{F'}_{j\ell}(k),\quad j\in[a],\ \ell\in[b],\ |k|> N\ee
and
\be\label{S''}\hat{S'}_{j\ell}(k)=\frac i{\langle k,\om(\r)\rangle \ -\alpha_j(\r)-\beta_\ell(\r)}\hat{F'}_{j\ell}(k),\quad j\in[a],\ \ell\in[b],\ |k|\leq N\ee
where $\alpha_j(\r)$ and $\beta_\ell(\r)$ denote eigenvalues of $Q_{[a]}(\r)$ and $Q_{[b]}(\r)$, respectively. First notice that by \eqref{R''} one has
 $$|R(\theta)|_{s,\b}=|R'(\theta)|_{s,\b}\leq \frac{Ce^{-\frac12 (\s-\s')N}}{ (\s-\s')^{n}}
\sup_{|\Im\theta|<\s}| F(\theta)|_{s,\b}.$$
To solve \eqref{S''} we face the  small divisors
\be\label{divisors}
\langle k,\omega(\rho)\rangle +\alpha_j(\rho)-\beta_\ell(\rho),\quad j\in [a],\ \ell\in [b].
\ee
To estimate them, we have to distinguish between the case  $k= 0$ and $k\neq 0$.

\medskip

{\it The case $k=0$.}
In that case   we know that $w_a\neq w_b$ 
and we use   \eqref{ass2} and \eqref{la-lb} to get 
$$|\alpha_j(\r)-\beta_\ell(\r)|\geq c_0|w_a-w_b|-\frac{\de_0}{4w_a^{2\b}}
-\frac{\de_0}{4w_b^{2\b}}\geq \ka(1+|w_a-w_b|).$$
This last estimate allows us to use Lemma \ref{delort} to conclude that 
$$
| \hat S(0)|_{\b+}\leq \frac{C}{\ka^{d^{*}+1}}| \hat F(0)|_{\b}.$$

\medskip

{\it The case $k\not=0$. }  If $k\neq 0$ we face the small divisors \eqref{divisors} with 
non-trivial $\langle k,\omega\rangle$. 
 Using Hypothesis A3, there is a set
 $\D_2'=\D(\om, 2\eta, N)$, 
 $$\meas(\D\setminus {\D_2'})\leq CN^{\a_1}(\frac{\eta}{\delta_0})^{\a_2} ,$$
 such that for all $\r\in\D_2'$ and  $0<|k|\leq N$ 
  $$|\langle k, \om(\r)\rangle \ -\la+\lb|\geq 2\eta(1+|w_a-w_b|).$$
 By \eqref{ass2} this implies
 \begin{align*}|\langle k, \om(\r)\rangle \ -\alpha_j(\r)+\beta_\ell(\r)|&\geq 2\eta(1+|w_a-w_b|)-\frac{\de_0}{4w_a^{2\b}}
 -\frac{\de_0}{4w_b^{2\b}}\\
& \geq \eta(1+|w_a-w_b|)\end{align*}
 if
$$ w_b\geq w_a\geq \Big( \frac{\de_0}{2\eta}\Big)^{\frac1{2\b}}.$$
Let now $ w_a\leq ( \frac{\de_0}{2\eta})^{\frac1{2\b}}$. We note that $|\langle k, \om(\r)\rangle \ -\la+\lb|\leq 1$ implies that 
$$w_b^{\delta}\leq( \frac{\de_0}{2\eta})^{\frac\delta{2\b}}+ C|k|\leq ( \frac{\de_0}{2\eta})^{\frac\delta{2\b}}+ N\,.$$
As in Section 5.3, we obtain that
\be\label{inversebis}
 |\langle k,\omega(\rho)\rangle +\alpha_j(\rho)-\beta_\ell(\rho)|\geq {\ka}(1+|w_a-w_b|)\quad \forall j\in[a],\ \forall \ell\in[b]\ee
 holds
outside a set $F_{[a],[b],k}$ of measure $w_a^{d^{*}}w_b^{d^{*}}(1+|w_a-w_b|)\ka\de_0^{-1}$. This can be done considering equation \eqref{S''} as the multiplication of a vector of size $d_{[a]}d_{[b]}$ called $\hat{F}'_{jl}(k)$ by a real diagonal (hence hermitian) square $d_{[a]}d_{[b]} \times d_{[a]}d_{[b]}$ matrix, and using Hypothesis A2, Condition \eqref{ass2} and Lemma \ref{EK}.

If $F$ is the union of  $F_{[a],[b],k}$ for $|k|\leq N$, $[a],[b]\in\hat\L$ such that $ w_a\leq ( \frac{\de_0}{2\eta})^{\frac1{2\b}}$ and $w_b^{\delta} \leq ( \frac{\de_0}{2\eta})^{\frac\delta{2\b}}+ N$ respectively, we have 
 \begin{align*}\meas(F)&\leq
C( \frac{\de_0}{2\eta})^{\frac{d^{*}+1}{2\b}}\big(( \frac{\de_0}{2\eta})^{\frac\delta{2\b}}+ N\big)^{(d^{*}+2)/\delta}\frac\ka{\de_0}N^n\\
&\leq CN^{n+(d^{*}+2)/\delta} ( \frac{\de_0}{\eta})^{\frac{2d^{*}+3}{2\b}}\frac\ka{\de_0}\,.
\end{align*}
Now we choose $\eta$ so that
$$(\frac{\eta}{\delta_0})^{\a_2}=( \frac{\de_0}{\eta})^{\frac{2d^{*}+3}{2\b}}\frac\ka{\de_0}\quad \text{i.e. }\frac\eta{\de_0}=\big(\frac\ka{\de_0}\big)^{\frac{2\b}{2d^{*}+3+2\b \a_2}}.
$$
Then, as $\b\leq 1$, $\eta\leq \ka$ and $\delta \geq 1$, we have
$$\meas(F)\leq CN^{n+d^{*}+2} \big(\frac\ka{\de_0}\big)^{\frac{2\b\a_2}{2d^{*}+3+2\b\a_2}}\,.$$
Let $\D_3=\D_2\cap\D_2'\setminus F$, we have
$$\meas(\D\setminus\D_3)\leq CN^{exp}\big(\frac\ka{\de_0}\big)^{\frac{2\b\a_2}{2d^{*}+3+2\b\a_2}}$$
and by construction  for all $\r\in\D_3$,
$0<|k|\le N$, $a,b\in\L$ and $j\in[a],\ \ell\in[b]$ we have
$$|\langle k, \om(\r)\rangle \ -\alpha_j(\r)+\beta_\ell(\r)| \geq \ka(1+|w_a-w_b|).$$
 Hence using Lemma \ref{delort} once again we obtain from \eqref{homo+-} that $\hat S'(k)\in\msb^+$ and
 $$|\hat S'(k)|_{s,\b+}\leq C\ka^{-1-d^{*}/2\delta}N^{d^{*}/2\gamma}|\hat F'(k)|_{s,\b}.$$
Therefore we obtain a solution  $S$ satisfying for any $|\Im\theta|<\s'$
\begin{align*}
| S(\theta)|_{s,\b+}\leq &\frac{CN^{d^{*}/2\gamma}}{\ka^{1+d^{*}/2\delta} (\s-\s')^{n}}
\sup_{|\Im\theta|<\s}| F(\theta)|_{s,\b},\\
\end{align*}
The estimates of the derivatives with respect to $\r$ are obtained by differentiating \eqref{homo-} and proceeding as at the end of the resolution of equation \eqref{homo4-2}.

In this way we have constructed a solution $S_{\zeta \zeta}, R_{\zeta \zeta}, B$  of 
the fourth component of the homological equation which satisfies all required estimates.
To guarantee that it is real, as at the end of Section~\ref{s5.3} we replace  $S_{\zeta \zeta}, R_{\zeta \zeta}, B$
 by their real parts and extend it analytically to $\T^n_{\sigma'}$ 
(e.g, replace $S_{\zeta \zeta}(\theta,\rho)$ by $\frac12(S_{\zeta \zeta}(\theta,\rho) + 
\bar S_{\zeta \zeta}( \bar\theta,\rho) )$).
 \endproof

\subsection{Summing up}

Let
$$h=\om(\r)\cdot r +\frac 1 2 \langle \zeta,  A(\r)\zeta\rangle$$
where $\r\to\om(\r)$ and $\r\to A(\r)$ are $C^1$ on $\D$ and $A$ is on normal form.
\begin{proposition}\label{thm-homo}

Assume
\be\label{Aom}
 |\partial_\r^j (A(\r)-A_0)|_{s,\b} \le \frac{\delta_0}{4},\quad |\partial_\r^j(\om-\om_0)|\leq \delta_0\ee
for $j=0,1$  and $\r\in \D$.
Let $f\in \Tc^{s,\b}(\s,\mu,\D)$, $0<\ka\leq \frac{\de_0}2$ and $N\ge 1$. 
 Then there exists a subset $\D'=\D'(h, \ka,N)\subset \D$, satisfying
 $$\meas (\D\setminus \D')\leq  
 CN^{\exp} \Big(\frac{\ka}{\delta_0}\Big)^{\exp'}, $$
 and there exist real jet-functions 
   $S\in\Tc^{s,\b+}(\s',\mu,\D')$ , $R\in\Tc^{s,\b}(\s',\mu,\D')$ and a  normal form
$$\hat h=\( f(\cdot,0;\r) \)+\( \nabla_r f(\cdot,0;\r) \)\cdot r+ \frac 1 2 \langle \zeta, B(\r)\zeta\rangle,$$
  such that
$$
\{ h,S \}+f^T=\hat h+R.
$$
Furthermore, for all  $0\le\s'<\s$
\be\label{estim-B}|\partial_\r^jB(\r)|_{s,\b}\le
C\ [f]^{s,\b}_{\s,\mu,\D'},\quad j=0,1 \text{ and }\r\in\D'
\ee
\be\label{estim-S}
[S ]^{s,\b+}_{\s',\mu,\D'} \leq 
C\frac{N^{1+d^{*}/\gamma}}{\ka^{2+d^{*}/2\beta} (\s-\s')^{n}}[f]^{s,\b}_{\s,\mu,\D'}\ee
 
\be\label{estim-R}
[R]^{s,\b}_{\s',\mu,\D'}\leq 
C
\frac{e^{-\frac12(\s-\s')N}}{ (\s-\s')^{n}}[f]^{s,\b}_{\s,\mu,\D'}.\ee
The two exponents $\exp$ and $\exp'$ are positive numbers depending on $c_{0}$, $n$, $d^{*}$, $\a_1$, $a_2$, $\ga$, $\b$. The constant $C$  also depends  on $|\om_{0}|_{C^{1}(\D)}$.
 
\end{proposition}
\proof
We define  $S$ by
$$
S(\theta,r,\zeta)=S_\theta(\theta)+\langle S_r(\theta), r\rangle+\langle S_\zeta(\theta),\zeta\rangle+
\frac 1 2 \langle S_{\zeta\zeta}(\theta)\zeta,\zeta \rangle.
$$ 
where $S_\theta$, $S_r$, $S_\zeta$ and $S_{\zeta\zeta}$ are constructed in Propositions \ref{prop:homo12}, \ref{prop:homo3} and \ref{prop:homo4}. Hamiltonians $R$ and $B$ are also constructed in these 3 propositions. Then all the statements in Proposition \ref{thm-homo} are satisfied and in particular we notice that
$$\nabla_\zeta S= S_\zeta+S_{\zeta\zeta}\zeta$$
belongs to $Y_{s+\b}$ as a consequence of Propositions \ref{prop:homo3}, \ref{prop:homo4} and Lemma \ref{product} (iii).
\endproof

\section{Proof of the KAM Theorem.}
The Theorem \ref{main} is proved by an iterative KAM procedure. We first describe the general step of this KAM procedure.
\subsection{The KAM step}
Let $h$ be a normal form Hamiltonian
$$
h= \om\cdot r +\frac 1 2 \langle \zeta, A(\om)\zeta\rangle$$
with $A$ on normal form, $A-A_0\in\M_\b$ and satisfying \eqref{Aom}. Let $f\in \Tc^{s,\b}(\s,\mu,\D)$ be a (small) Hamiltonian perturbation.  
Let $S=S^T \in \Tc^{s,\b+}(\s',\mu,\D')$ be the  solution of the homological equation
\be \label{eq-homobis}
\{ h,S \}+f^T=\hat h+R.
\ee
defined in Proposition \ref{thm-homo}.
Then defining 
 $$ h^+:=h+\hat h,$$
we get
$$h\circ \Phi^1_S=h^++f^+$$
with
\be \label{f+}
f^+= R+(f-f^T)\circ \Phi^1_S+\int_0^1\{ (1-t)(\hat h+R)+tf^T,S \}\circ \Phi^t_S\ \dd t.
\ee
The following Lemma gives an estimation of the new perturbation:

\begin{lemma}\label{lem-f+}
Let $\ka>0$, $N\geq 1$, $0<\s'<\s\leq 1$ and $0<2\mu'< \mu\leq 1$. Assume that $\D'\subset \D$, that $f\in\Tc^{s,\b}(\s,\mu,\D)$, that $R$ satisfies \eqref{estim-R}  and that $S=S^T$ belongs to $\Tc^{s,\b+}(\s'',\mu,\D')$ with $\s''=\frac{\s+\s'}{2}$ and satisfies
\begin{equation}
\label{hypo-S} 
[S]^{s,\b+}_{\s'',\mu,\D'}\leq \frac 1{16} \mu^2 ( \s- \s').
\end{equation}
Then the function $f^+$ given by formula \eqref{f+} belongs to $ \Tc^{s,\b}(\s',\mu',\D')$ and
\begin{align}\begin{split}\label{estim-f+}
[f^+]^{s,\b}_{\s',\mu',\D'}\leq M\left(  \frac{e^{-\frac 1 2(\s-\s')N}}{  (\s-\s')^{n}}+ \left( \frac{\mu'}{\mu}\right)^3 \right. 
+\left.  \frac{N^{1+d^{*}/\gamma}}{\ka^{2+d^{*}/2\beta}\mu^2 (\s-\s')^{n+1}}[f]^{s,\b}_{\s,\mu,\D}\right) [f]^{s,\b}_{\s,\mu,\D}
\end{split}\end{align}
where $M$ is a constant depending on $n$, $d^{*}$, $\alpha_{1}$, $\alpha_{2}$, $c_{0}$, $\gamma$ and $\beta$.
\end{lemma}
\proof
Let us denote the three terms in the r.h.s. of \eqref{f+} by $f^+_1$, $f^+_2$ and $f^+_3$. In view of \eqref{estim-R}, we have that $[f^+_1]^{s,\beta}_{\s',\mu',\D'}$ is controlled by the first term in r.h.s. of \eqref{estim-f+}. \\
By Proposition \ref{lemma:jet}, we get 
$$[f-f^T]^{s,\b}_{\s,2\mu',\D'}\leq C \left( \frac{\mu'}{\mu}\right)^3 [f]^{s,\b}_{\s,\mu,\D}.$$
By hypothesis $S=S^T $ belongs to $ \Tc^{s,\b+}(\s',\mu,\D')$ and satisfies \eqref{hypo-S} which implies $[S]^{s,\b+}_{\s'',\mu,\D'}\leq \frac 1{2} (\mu-\mu')^2 ( \s''- \s')$ since $2\mu'<\mu$. Therefore 
 by Lemma \ref{composition} and since $2\mu'\leq 2 (\mu-\mu')$, $[f^+_2]^{s,\b}_{\s',\mu',\D'}$ is controlled by the second term in r.h.s. of \eqref{estim-f+}. \\
 It remains to control $[f^+_3]^{s,\b}_{\s',\mu',\D'}$. To begin with, $g_t:=(1-t)(\hat h+R)+tf^T$ is a jet function in $ \Tc^{s,\b}(\s',\mu,\D)$. Furthermore, defining  for $j=1,2$,
$$\s_j=\s'+j\frac{\s-\s'}{3} $$
 and  using \eqref{estim-R} we get (for $N$ large enough)
$$
[g_t]^{s,\b}_{\s_2,\mu,\D'}\leq C\left( 1+3^{n}\frac{e^{-(\s-\s')N/6}}{  (\s-\s')^{n}}\right)[f]^{s,\b}_{\s,\mu,\D}\leq C [f]^{s,\b}_{\s,\mu,\D}.
$$
On the other hand $ S\in \Tc^{s,\b+}(\s_2,\mu,\D')$ is also a jet function and satisfies 
$$
[S]^{s,\b+}_{\s_2,\mu,\D'}\leq \frac{CN^{1+d^{*}/\gamma}}{\ka^{2+d^{*}/2\beta} (\s-\s')^{n}}[f]^{s,\b}_{\s,\mu,\D}.
$$
Then using Lemma \ref{lemma-poisson} we have 
$$
[\{g_t,S\}]^{s,\b}_{\s_1,\mu,\D'}\leq C\frac{N^{1+d^{*}/\gamma}}{\ka^{2+d^{*}/2\beta}\mu^2 (\s-\s')^{n+1}}([f]^{s,\b}_{\s,\mu,\D})^2.$$
We conclude the proof by Proposition \ref{composition}.
\endproof
\subsection{Choice of parameters}
To prove the main theorem we  construct the transformation $\Phi$ as the composition of infinitely many transformations $S$ as in Theorem \ref{thm-homo}, i.e. for all $k\geq 1$ we construct iteratively $S_{k}$, $h_k$, $f_k$ following the general scheme \eqref{eq-homobis}--\eqref{f+} as follows~:
$$(h+f)\circ \Phi^1_{S_{1}}\circ\cdots\circ \Phi^1_{S_k}= h_{k}+f_{k}.$$
At each step $f_k\in \Tc^{s,\b}(\s_k,\mu_k,\D_{k})$ with $[f_k]^{s,\b}_{\s_k,\mu_k,\D_{k}}\leq \eps_k$ , $h_k=\lan\om_k,r\ran +\frac 12\lan \zeta,A_k\zeta\ran$ is on normal form, the Fourier series are truncated at order $N_k$ and the small divisors are controlled by $\ka_k$. In this section we specify the choice of all the parameters for $k\geq 1$. \\
First we fix
$$ \ka_0=\eps^{\frac 1{24(2+d^{*}/2\beta)}}.$$
 We define  $\eps_0=\eps$, $\s_0=\s$,  $\mu_0=\mu$ and for $j\geq 1$ we choose
 \begin{align*}
 \s_{j-1}-\s_j=&C_* \s_0 j^{-2},\\
 N_j=&2(\s_{j}-\s_{j+1})^{-1}\ln \eps_j^{-1},\\
 \ka_{j} =& \eps_{j}^{\frac{1}{24(2+d^{*}/2\beta)}}\\
 \mu_{j} =& \left( \frac{\eps_{j}}{(2M)^{j}\eps^{6/5}}\right)^{\frac13}\,,
 \end{align*}
where $M$ is the absolute constant defined in \eqref{estim-f+} and $(C_*)^{-1} =2\sum_{j\geq 1}\frac 1{j^2}$, and
\begin{equation}
 \eps_{j} = (\eps_{j-1})^{\frac54} \,.\label{eps+}
 \end{equation}
Observe that with this choice, $(\mu_{j})$ satisfies $2\mu_{j+1} \leq \mu_{j}$. Then the only unfixed parameter is $\eps=\eps_{0}$, that will be fixed next section. Nevertheless,  $\eps$ will be small enough to ensure the property
 $ \ka_{j} \leq \frac{\delta_{0}}{2}\,$
 that is necessary to apply Proposition \ref{thm-homo}. This is guaranteed if 
 \be\label{ed} \eps^{\frac 1{24(2+d^{*}/2\beta)}} \leq \frac{\delta_{0}}{2}\,.\ee

\subsection{Iterative lemma}

Let set $\D_0=\D$, $h_0=\lan\om_0(\r), r\ran+\frac 1 2 \langle \zeta, A_0\zeta\rangle$ and  $f_0= f$ in such a way $[f_0]^{s,\b}_{\s_0,\mu_0,\D_{0}}\leq \eps_0$. For $k\geq 0$ let us denote
$$\O_{k}= \O^{s}(\s_{k},\mu_{k}).$$
\begin{lemma}\label{iterative} For $\eps$ sufficiently small depending on $\mu_0$, $\s_0$, $n$,$s$, $\b$ and $|\om_{0}|_{C^{1}(\D)}$ we have the following:\\
For all $k\geq 1$ there exist $\D_k\subset\D_{k-1}$, $S_{k}\in \Tc^{s,\b+}(\s_k,\mu_k,\D_{k})$, $h_k=\lan\om_k,r\ran +\frac 12\lan \zeta,A_k\zeta\ran$  on normal form and 
$f_k\in \Tc^{s,\b}(\s_k,\mu_k,\D_{k})$ such that
\begin{itemize}
\item[(i)]  The mapping \be \label{Phik} \Phi_{k}(\cdot,\r)=\Phi^1_{S_{k}}\ :\ \O_k\to \O_{k-1}, \quad \r\in \D_{k},\ k=1,2,\cdots\ee
is an analytic symplectomorphism linking the hamiltonian at step $k-1$ and the hamiltonian at the step  k, i.e.
$$(h_{k-1}+f_{k-1})\circ \Phi_{k}= h_k+f_k.$$
\item[(ii)] we have the estimates
\begin{align*}
\meas(\D_{k-1}\setminus \D_{k})&\leq \eps_{k-1}^\a,\\
[h_k-h_{k-1}]^{s,\b}_{\s_k,\mu_k,\D_{k}}&\leq C\eps_{k-1},\\
[f_k]^{s,\b}_{\s_k,\mu_k,\D_{k}}&\leq \eps_k,\\
\| \Phi_k(x,\r)-x\|_s&\leq \eps^{4/5} .\eps_{k-1}^{1/4},\ \text{ for } x\in \O_k,\ \r\in\D_k.
\end{align*}
\end{itemize}
The  exponents $\a$ is a positive number depending on $n$, $d^{*}$, $\a_1$, $a_2$, $\ga$, $\b$. The constant $C$  also depends on $|\om_{0}|_{C^{1}(\D)}$.
 \end{lemma}
 \proof At step 1, $h_0=\lan\om_0(\r), r\ran+\frac 1 2 \langle \zeta, A_0\zeta\rangle$ and thus hypothesis \eqref{Aom} is trivially satisfied and we can apply Proposition \ref{thm-homo} to construct $S_1$, $R_0$, $B_0$ and $\D_1$ such that for $\r\in\D_1$
 $$\{h_0,S_0\}+f_0^T=\hat h_0+R_0.$$
Then we see that, using \eqref{estim-S} and defining  $\s_{1/2}=\frac{\s_0+\s_1}2$, we have
$$ [S_{1}]^{s,\b+}_{\s_{1/2},\mu_0,\D_{1}}\leq C \frac{\eps_0 N_0^{1+d^{*}/\gamma}}{\ka_0^{2+d^{*}/2\beta} (\s_0-\s_{1/2})^n}\leq\frac 1{16} \mu_0^2 ( \s_0- \s_{1})$$
for $\eps=\eps_0$ small enough in view of our choice of parameters.
 Therefore both Proposition \ref{Summarize} and Lemma \ref{f+} apply  and thus for any $\r\in\D_1$,
 $\Phi_{1}(\cdot,\r)=\Phi^1_{S_{1}} :  \O_1\to \O_0$ is an analytic symplectomorphism such that
 $$(h_{0}+f_{0})\circ \Phi_{1}= h_1+f_1$$ with  $h_1$, $f_1$, $\D_1$ and $\Phi_1$  satisfying the estimates $(ii)_{k=1}$ . In particular we have
 $$\|\Phi_1(x)-x\|_s\leq \frac C{\s_0\mu_0^2}[S_{1}]^{s,\b+}_{\s_{1/2},\mu_0,\D_{1}}\leq \frac {CN_0^{1+d^{*}/\gamma}}{\s_0^{n+1}\mu_0^2\ka_0^{2+d^{*}/2\beta}}\eps_0\leq \frac {C(\ln\eps_0)^{1+d^{*}/\gamma}}{\s_0^{n+2+d^{*}/\gamma}\mu_0^2}\eps_0^{23/24}\leq \frac 12 \eps_0^{11/12}     $$
 for $\eps_0$ small enough. 
 
 \medskip
 
 Now assume that we have completed the iteration up to step $j$.
We want to perform the step $j+1$. We first note that by construction (see Proposition \ref{thm-homo})
$$A_j =A_0+B_0+\cdots+B_{j-1}$$
and by \eqref{estim-B}
$$|A_j|_\b \leq \eps_0+\cdots +\eps_{j-1}\leq 2\eps_0\leq \frac 14 \delta_0$$ 
for $\eps_0$ small enough.
Similarly
$$\om_j=\om_0+\( \nabla_r f_0(\cdot,0;\r) \)+\cdots +\( \nabla_r f_{j-1}(\cdot,0;\r) \)$$
and thus $|\partial_r^j(\om_j-\om_0)|\leq \delta_0$ for $\eps_0$ small enough.\\ Therefore \eqref{Aom} is satisfied at rank $j$ and we can apply
Proposition \ref{thm-homo} in order to construct $S_{j+1}$, $B_j$, $R_j$ and $\D_j$.  

Then we construct  $f_{j+1}$ as in \eqref{f+}, i.e.
$$
f_{j+1}= R_j+(f_j-f_j^T)\circ \Phi^1_{S_{j+1}}+\int_0^1\{ (1-t)(\hat h_j+R_j)+tf_j^T,S_{j+1} \}\circ \Phi^t_{S_{j+1}}\ \dd t.
$$
To control $f_{j+1}$ we may apply Lemma \ref{lem-f+} since, defining $\s_{j+1/2}=\frac{\s_j+\s_{j+1}}2$,
$$
[S_{j+1}]^{s,\b+}_{\s_{j+1/2},\mu_j,\D_{j+1}}\leq  C \frac{\eps_jN_j^{1+d^{*}/\gamma}}{\ka_j^{2+d^{*}/2\beta} (\s_j-\s_{j+1})^n}  \leq \frac 1{8} \mu_j^2 ( \s_j- \s_{j+1}).$$
Therefore we can apply  Lemma \ref{lem-f+} and, using the preceding choice of parameters, we may bound all the terms of the r.h.s. of \eqref{estim-f+}. Let us start with the second term:
\begin{equation}\label{epsy2}
 M \left( \frac{\mu_{j+1}}{\mu_{j}}\right)^{3}\eps_{j} = \frac{1}{2} \eps_{j+1}\,.
 \end{equation}
The third term may be computed as
\begin{equation}\label{epsy3}
M\left( \frac{ 2(j+1)^{2}\ln(\eps_{j}^{-1})}{C_{*}\sigma_{0}}\right)^{1+d^{*}/\gamma}\left(\frac{(j+1)^{2}}{C_{*}\sigma_{0}}\right)^{n+1} \frac{\eps_{j}^{2- 1/24}}{\mu_{j}^{2}} = C (j+1)^{2n+3+2d^{*}/\gamma}(2M)^{2j/3}\eps^{4/5} (\eps_{j})^{1/24}\eps_{j+1} 
\end{equation}
and there exists $\bar{\eps}_{1}>0$ such that for $0 <\eps \leq \bar{\eps}_{1}$ we have for any $j \geq 1$ 
$$ C (j+1)^{2n+3+2d^{*}/\gamma}(2M)^{2j/3}(\eps)^{\frac 45+\frac{1}{24}.(\frac 54)^{j}} \leq \frac{1}{4}\,.$$
The first term gives
\begin{equation}\label{epsy1}
M \frac{\eps_{j}^{2}}{C_{*}\sigma_{0}}(j+1)^{2n} = M \frac{(j+1)^{2n}}{C_{*}\sigma_{0}}(\eps)^{\frac 34. (\frac 54)^{j}}\eps_{j+1}\,,
\end{equation}
and there exists $\bar{\eps}_{2}>0$ such that for $0 <\eps \leq \bar{\eps}_{2}$ we have for any $j \geq 1$ 
$$ M \frac{(j+1)^{2n}}{C_{*}\sigma_{0}}(\eps)^{\frac 34. (\frac 54)^{j}} \leq \frac{1}{4}\,.$$
Take $\eps_{0}\ \leq \bar{\eps} = \min(\bar{\eps}_{1}, \bar{\eps}_{2})>0$ and we conclude that
\be \label{recu}
[f_{j+1}]^{s,\b}_{\s_{j+1},\mu_{j+1},\D_{j+1}}\leq \eps_{j+1}.\ee
On the other hand by Proposition \ref{thm-homo} the domain $\D_{j+1}$ satisfies
$$\meas (\D_{j}\setminus\D_{j+1})\leq CN^{\exp}_{j}\Big(\frac{\ka_j}{\delta_0}\Big)^{\exp'}\leq \eps_{j}^\alpha $$
 for some $\a>0$ and for $\eps_0=\eps$ small enough. The estimate concerning $h_{k+1}-h_{k}$ follows from \eqref{estim-B} and \eqref{recu} for the infinite dimensional part, from \eqref{recu} for the control of $\( f_{j+1}(\cdot,0;\r) \)$ and a straightforward Cauchy estimate for the control of the mean value $\( \nabla_r f_{j+1}(\cdot,0;\r) \)$. Concerning the flow, we have for $j \geq 1$, 
 \begin{eqnarray*}
 \|\Phi_{j+1}(x)-x\|_s&\leq& \frac C{\s_j\mu_j^2}[S_{j+1}]^{s,\b+}_{\s_{j+1/2},\mu_j,\D_{j+1}}\leq \frac {CN_j^{1+d^{*}/\gamma}}{\s_j^{n+1}\mu_j^2\ka_j^{2+d^{*}/2\beta}}\eps_j\\&\leq& {C'(\ln\eps_j)^{1+d^{*}/\gamma}(2M)^{2j/3}j^{n+2+d^{*}/\gamma}}\eps^{4/5}\eps_j^{7/24}\leq \eps^{4/5}\frac 12 \eps_j^{1/4}\,, \end{eqnarray*}
  for $\eps$ small enough.

\endproof

\subsection{Transition to the limit and proof of Theorem \ref{main}}

Let $$\D'=\cap_{k\geq 0}\D_k.$$ In view of the iterative lemma, this is a Borel set satisfying
$$\meas(\D\setminus\D')\leq 2 \eps ^\a.$$
Let us set
$$Q_l=\O^{s}(\s/\ell,\mu/\ell), \ \mathcal Z_s=\T_\s^n\times\C^n\times Y_s$$
where $\ell\geq 2$, and recall that  $\|\cdot\|_s$ denotes the natural norm on $\C^n\times\C^n\times Y_s$. It defines the distance on $\mathcal Z_s$. We used the notations introduced in Lemma \ref{iterative}. By Proposition \ref{changevar} assertion 2 and since $\s_k>\s/2$, for each $\r\in \D'$ and $k \geq 2$, the map $\Phi_k$ extends to $Q_2$ and satisfies on $Q_2$ the same estimate as on $\O_k$:
\be \label{estim-Phik}\Phi_k:\ Q_2\to \mathcal Z_s,\quad \|\Phi_k-\mathrm{Id}\|_s\leq C\mu_k^{-2}(\s_{k-1}-\s_k)^{-1}\eps_k\leq \frac{k^{2}}{C_{*}\sigma_{0}}(2M)^{2k/3}\eps_k^{1/3}\eps^{4/5}.\ee
Now for $0\leq j\leq N$ let us denote $\Phi^j_N=\Phi_{j+1}\circ\cdots\circ \Phi_N$. Due to \eqref{Phik}, it maps $\O_N$ to $\O_{j}$. Again using  Proposition \ref{changevar}, this map extends analytically to a map $\Phi^j_N:\ Q_2\to \mathcal Z_s,$ and by \eqref{estim-Phik}, for $M>N$, $\|\Phi_N^j-\Phi_M^j\|_s\leq C \eps_N^{1/4}\eps^{4/5} $, i.e. $(\Phi_N^j)_N$ is a Cauchy sequence. Thus when $N\to \infty$ the maps  $\Phi^j_N$ converge to a limiting mapping $\Phi_\infty^j:\ Q_2\to \mathcal Z_s.$ Furthermore we have
\be \label{estim-Phiinf}\|\Phi_\infty^j-\mathrm{Id}\|_s\leq C\eps^{4/5}\sum_{k\geq j} \eps_k^{1/4}\leq C\eps^{4/5}\eps_j^{1/4},\ \forall j\geq 1.\ee
By the Cauchy estimate the linearized map satisfies
\be \label{linearized} \|D\Phi_\infty^j(x)-\mathrm{Id}\|_{\L(Y_s,Y_s)}\leq C\eps^{4/5}\eps_j^{1/4},\quad \forall x\in Q_3,\ \forall j\geq 1.\ee
By construction, the map $\Phi_N^0$ transforms the original hamiltonian $$H_0=\lan\om, r\ran+\frac 1 2 \langle \zeta, A_0\zeta\rangle+f$$ into $$H_N=\lan\om_N, r\ran+\frac 1 2 \langle \zeta, A_N(\om)\zeta\rangle+f_N.$$
Here $$\om_N=\om+\( \nabla_r f_0(\cdot,0;\r) \)+\cdots +\( \nabla_r f_{N-1}(\cdot,0;\r) \)$$ and $$A_N= A_0+B_0+\cdots+B_{N-1}$$ where 
$B_k$ is built from $\langle\nabla^2_{\zeta\zeta}f_k(\cdot,0)\rangle$ as in the proof of Proposition \ref{prop:homo4}.\\
Clearly, $\om_N\to \om'$ and $A_N\to A$ where the vector $\om'\equiv\om'(\r)$ and the operator $A\equiv A(\r)$ satisfy the assertions of Theorem \ref{main}.\\
Let us denote $\Phi=\Phi_\infty^0$, consider the limiting hamiltonian $H'=H_0\circ \Phi$ and write it as
$$ H'= \lan\om', r\ran+\frac 1 2 \langle \zeta, A(\r)\zeta\rangle+f'.$$
The function $f'$ is analytic in the domain $Q_2$. Since $H'=H_k\circ \Phi_\infty^k$,  we have
$$\nabla H'(x)=D\Phi_\infty^k(x)\cdot\nabla H_k(\Phi_\infty^k(x)).$$
As $[f_k]^{s,\b}_{\s_k,\mu_k,\D_k}\leq \eps_k$, we deduce 
 $$\nabla_r H_k(\Phi_\infty^k(\theta,0,0))=\om_k+O(\eps_k^{1/4})\quad \theta\in\T^n_{\frac \s 3}.$$ 
 Since the map $\Phi_\infty^k$ satisfies \eqref{linearized}, then $$\nabla_r H'(\theta,0,0)=\om'+O(\eps_k^{1/4}) \quad \text{ for all } k\geq 1\text{ and } \theta\in\T^n_{\frac \s 3} .$$ Hence, $\nabla_r H'(\theta,0,0)=\om'$ and thus
$$\nabla_r f'(\theta,0,0)\equiv 0\quad \text{ for } \theta\in\T^n_{\frac \s 3} .$$
Similar arguments lead to
$$\nabla_{\zeta_a} f'(\theta,0,0)\equiv 0 \text{ and }\nabla_{\zeta_a}\nabla_r f'(\theta,0,0)\equiv 0\quad \text{ for } \theta\in\T^n_{\frac \s 3} .$$
Now consider $\nabla_{\zeta_a}\nabla_{\zeta_b}H'(x)$. To study this matrix let us write it in the form \eqref{feo}, with $h=H_k$ and $x(1)=\Phi_\infty^k(x)$. Repeating the arguments used in the proof of Proposition \ref{composition} we get that
$$\nabla_{\zeta_a}\nabla_{\zeta_b}H'(\theta,0,0)=(A_k)_{ab}+O(\eps_k^{1/4}) \quad \text{ for all } k\geq 1\text{ and } \theta\in\T^n_{\frac \s 3} .$$
Therefore $\nabla_{\zeta_a}\nabla_{\zeta_b}H'(\theta,0,0)= A_{ab}$ i.e. 
$$\nabla_{\zeta_a}\nabla_{\zeta_b}f'(\theta,0,0)=0 \quad \text{ for } \theta\in\T^n_{\frac \s 3} .$$

This concludes the proof of Theorem \ref{main}.

\appendix
\section{Some calculus}
\begin{lemma}\label{lem-A1}
Let $j,k,\ell\in\N\setminus\{0\}$ then
\be\label{AA}\frac{\min(j,k)}{\min(j,k)+|j^2-k^2|}\frac{\min(k,\ell)}{\min(k,\ell)+|k^2-\ell^2|}\leq \frac{\min(j,\ell)}{\min(j,\ell)+|j^2-\ell^2|}.\ee
\end{lemma}
\proof
Without lost of generality we can assume $j\leq\ell$.\\
If $k\leq j$ then $|k^2-\ell^2|\geq |j^2-\ell^2|$  and thus
\begin{align*}\frac{\min(j,\ell)}{\min(j,\ell)+|j^2-\ell^2|}&=\frac{j}{j+|j^2-\ell^2|}\geq \frac{j}{j+|k^2-\ell^2|}\\ &\geq \frac{k}{k+|k^2-\ell^2|}=\frac{\min(k,\ell)}{\min(k,\ell)+|k^2-\ell^2|}\end{align*}
which leads to \eqref{AA}. The case $\ell\leq k$ is similar.\\
In the case $j\leq k\leq \ell$ we have 
\begin{align*}\frac{\min(j,k)}{\min(j,k)+|j^2-k^2|}&\frac{\min(k,\ell)}{\min(k,\ell)+|k^2-\ell^2|}\leq \frac{j}{j+|j^2-k^2|+|k^2-\ell^2|}\\ 
&\leq \frac{j}{j+|j^2-\ell^2|}=\frac{\min(j,\ell)}{\min(j,\ell)+|j^2-\ell^2|}.\end{align*}
\endproof

\begin{lemma}\label{lem-A2}
Let $j\in\N$ then
$$\sum_{k\in \N}\frac{1}{k^\b(1+|k-j|)}\leq C$$
for a constant $C$ depending only on $\b>0$.
\end{lemma}
\proof
We note that $$\sum_{k\in \N}\frac{1}{k^\b(1+|k-j|)}=a\star b(j)$$ where  $a_k=\frac 1k$ for $k\geq1$, $a_k=0$ for $k\leq 0$ and $b_k=\frac 1{1+|k|}$, $k\in\Z$.
We have that $b\in\ell^p$ for any $1<p\leq+\infty$ and that $a\in\ell^q$ for any $\frac1\b<q\leq+\infty$. Thus by Young inequality $a\star b\in\ell_r$  for $r$ such that $\frac 1p+\frac 1q=1+\frac 1r$. In particular choosing $q=\frac 2\b$ and $p=\frac 2{2-\b}$ we conclude that  $a\star b\in\ell_\infty$.
\endproof
The following Lemma is a variant of Proposition 2.2.4 in \cite{DS1}.

 \begin{lemma}\label{delort} Let $A\in \msb$ and let $B(k)$ defined by
 \be\label{eqdelort}{B(k)}_{j}^l=\frac i{\langle k,\om \rangle \ +\eps\mu_j-\mu_l}{A}_j^l,\quad j\in[a],\ \ell\in[b]\ee
 where $\eps=\pm1$, $(\mu_a)_{a\in\L}$ is a sequence of real numbers satisfying
 \be \label{hypdelort0}|\mu_a-\lambda_a|\leq \min \left(\frac{C_{\mu}}{w_{a}^{\delta}}, \frac{c_{0}}{4} \right),\quad \text{ for all }a\in\L\ee
for a given $C_{\mu}>0$ and $\delta>0$,  and such that for all $a,b\in\L$ and all $|k|\leq N$
 \be\label{hypdelort} | \langle k,\om(\r)\rangle \ +\eps\mu_a-\mu_b|\geq {\ka}(1+|w_a-w_b|).\ee
 Then $B\in\msb^+$ and there exists a constant $C>0$ depending only on $C_{\mu}$, $\A$ and $\delta$  such that
 $$|B(k)|_{s,\b+}\leq C\frac{|A|_{s,\b}N^{\frac {d^{*}}{2\gamma}}}{\ka^{1+\frac{d^{*}}{2\delta}}}\quad \text{for all } |k|\leq N.$$
 \end{lemma}
\proof

We first remark that the claimed property only concerns the operator norms of the blocks $B_{[a]}^{[b]}$, which can be computed separately. Let $k_{1}$ and $k_{2}$ be positive integers that will be fixed later.  We define the following decomposition in $\msb$, according to the weights $w_{a}$ and $w_{b}$ :
$$ \msb = \Upsilon_{s,\beta}^{1}(k_{1},k_{2}) \oplus \Upsilon_{s,\beta}^{2}(k_{1},k_{2}) \oplus \Upsilon_{s,\beta}^{3}(k_{1},k_{2})\,,$$
where
\begin{align*}
\Upsilon_{s,\beta}^{1}(k_{1},k_{2})  &= \big\{ M \in \msb, M_{[a]}^{[b]} = 0 \;\; \mbox{if} \;\; \max(w_{a},w_{b}) \leq k_{1} \min(w_{a},w_{b}) \big\}\,,\\
\Upsilon_{s,\beta}^{2}(k_{1},k_{2})  &= \big\{ M \in \msb, M_{[a]}^{[b]} = 0 \;\; \mbox{if} \;\; \max(w_{a},w_{b}) > k_{1} \min(w_{a},w_{b}) \; \mbox{or} \;\max(w_{a},w_{b})\leq k_{2} \big\}\,,\\
\Upsilon_{s,\beta}^{3}(k_{1},k_{2})  &= \big\{ M \in \msb, M_{[a]}^{[b]} = 0 \;\; \mbox{if} \;\; \max(w_{a},w_{b}) > k_{1} \min(w_{a},w_{b}) \; \mbox{or} \;\max(w_{a},w_{b})> k_{2} \big\}\,,
\end{align*}
and we prove the desired estimates according to this decomposition. Since we estimate the operator norm of $B_{[a]}^{[b]}$, we need to rewrite the definition \eqref{eqdelort} in a operator way :  denoting by $D_{[a]}$ the diagonal (square) matrix with entries $\mu_{j}$, for $j \in [a]$ and $D'_{[a]}$ the diagonal (square) matrix with entries $\langle k,\om(\r)\rangle  +\eps\mu_j$, for $j \in [a]$, equation \eqref{eqdelort} reads
\be \label{eqdelort2} D'_{[a]} B_{[a]}^{[b]} -  B_{[a]}^{[b]} D_{[b]} =  i A_{[a]}^{[b]}\,.\ee

\noindent {\bf Step 1} : suppose  $A \in \msb \cap \Upsilon_{s,\beta}^{1}(k_{1},k_{2})$. The only nonzero blocks $A_{[a]}^{[b]}$ correspond to weights $w_{a}$ and $w_{b}$ such that 
$$ \max(w_{a},w_{b}) > k_{1} \min(w_{a},w_{b})\,$$
take for instance $w_{a} > k_{1} w_{b}$. Then $|w_{a}-w_{b}| \geq w_{a} (1 - \frac {1}{k_{1}})$, $w_a\geq k_1$ and 
\be
| \langle k,\om(\r)\rangle \ +\eps\mu_a| \geq  c_{0}\left(w_{a}^{\gamma} - \frac 14\right) - nN\max(\omega_{k}(\rho)) \geq \frac{c_{0}}{2} w_{a}^{\gamma}\,,\ee
for 
\be\label{k1} k_{1}\geq\left(4nNc_{0}^{-1}\max(\omega_{k}(\rho))\right)^{1/\gamma}:=C_{1} ,\ee
 that proves that $D'_{[a]}$ is invertible and gives an upper bound for the operator norm of its inverse. Then \eqref{eqdelort2} is equivalent to
\be B_{[a]}^{[b]} - {D'_{[a]}}^{-1} B_{[a]}^{[b]} D_{[b]} = i {D'_{[a]}}^{-1} A_{[a]}^{[b]}\,.\ee
Next consider the operator $\mathscr{L}^{1}_{[a]\times[b]}$ acting on matrices of size $[a]\times [b]$ such that
\be \mathscr{L}^{1}_{[a]\times[b]} \left( B_{[a]}^{[b]} \right) := {D'_{[a]}}^{-1} B_{[a]}^{[b]} D_{[b]}\,.\ee
We have 
\be \| \mathscr{L}^{1}_{[a]\times[b]} \left( B_{[a]}^{[b]} \right) \|  \leq  \frac{4 w_{b}}{w_{a }} \|B_{[a]}^{[b]}\| \leq  \frac{4}{k_{1}}  \|B_{[a]}^{[b]}\| \,,\ee
hence, in operator norm, $ \| \mathscr{L}^{1}_{[a]\times[b]} \| \leq \frac12$ if $k_{1}\geq 8$. Then the operator $\mathrm{Id} - \mathscr{L}^{1}_{[a]\times[b]}$ is invertible and
\begin{eqnarray*}
\| B_{[a]}^{[b]} \| &\leq& \| \left( \mathrm{Id} -  \mathscr{L}_{[a]\times[b]} \right)^{-1} \| \|  i {D'_{[a]}}^{-1} A_{[a]}^{[b]}\| \\
& \leq &  \frac{4}{w_{a}}  \|  A_{[a]}^{[b]}\| \\
& \leq &  \frac{4k_1}{k_1- 1} \frac{1}{1 + |w_{a}-w_{b}|}  \|  A_{[a]}^{[b]}\| 
\end{eqnarray*}
We have obtained that, for $k_{1}\geq\max(C_{1},8)$, $B \in \msb^{+}$ and
\be \label{delort1} |B|_{s,\b+} \leq 8|A|_{s,\b}\ee

\noindent {\bf Step 2} :  suppose  $A \in \msb \cap \Upsilon_{s,\beta}^{2}(k_{1},k_{2})$. The only nonzero blocks $A_{[a]}^{[b]}$ correspond to weights $w_{a}$ and $w_{b}$ such that 
$$ \max(w_{a},w_{b}) \leq k_{1} \min(w_{a},w_{b}) \; \mbox{and} \; \max(w_{a},w_{b})> k_{2} \,.$$
Notice that these two conditions imply that $$\min(w_a,w_b)\geq \frac{k_2}{k_1}.$$
We define the  square matrix $\tilde{D}_{[a]}=   \lambda_{a} \mathbf{1}_{[a]}$, where $\mathbf{1}_{[a]}$ is the identity matrix. Then 
\be \| D_{[a]} - \tilde{D}_{[a]} \| \leq \frac {C_{\mu}}{w_{a}^{\delta}}\,,\ee
and equation \eqref{eqdelort} may be rewritten as
\be \label{amel2} \mathscr{L}^{2}_{[a]\times[b]} \left( B_{[a]}^{[b]} \right) - \eps (\tilde{D}_{[a]} -  D_{[a]})B_{[a]}^{[b]}  + B_{[a]}^{[b]} (\tilde{D}_{[b]} -  D_{[b]})= A_{[a]}^{[b]}  \,,\ee
where we denote by $\mathscr{L}^{2}_{[a]\times[b]}$  the operator acting on matrices of size $[a]\times [b]$ such that
\be \mathscr{L}^{2}_{[a]\times[b]} \left( B_{[a]}^{[b]} \right) :=  \left( \langle k,\om(\r)\rangle + \eps \lambda_{a} - \lambda_{b}\right) B_{[a]}^{[b]} \,.
\ee
This dilation is invertible and \eqref{hypdelort} then gives, in operator norm,
\be \| \left(  \mathscr{L}^{2}_{[a]\times[b]} \right)^{-1}\| \leq  \frac{1}{\kappa (1 + |w_{a}-w_{b}|)}\,.\ee
This allows to write \eqref{amel2} as
\be \label{amel3} B_{[a]}^{[b]} - \left(  \mathscr{L}^{2}_{[a]\times[b]} \right)^{-1} \mathscr{K}_{[a]\times[b]}\left(B_{[a]}^{[b]} \right) =  \left(  \mathscr{L}^{2}_{[a]\times[b]} \right)^{-1}\left(A_{[a]}^{[b]}\right)\,,\ee
where $ \mathscr{K}_{[a]\times[b]}\left(B_{[a]}^{[b]} \right) = \eps (\tilde{D}_{[a]} -  D_{[a]})B_{[a]}^{[b]}  - B_{[a]}^{[b]} (\tilde{D}_{[b]} -  D_{[b]}) $. We have, thanks to \eqref{hypdelort0}, in operator norm,
\be \|\mathscr{K}_{[a]\times[b]} \| \leq C_{\mu}\left( \frac 1{w_{a}^{\delta}}+\frac 1{w_{b}^{\delta}}\right) \leq C_{\mu}\big(\frac{k_1}{k_{2}}\big)^{\delta}\,.\ee
Then for 
\be \label{k2}k_{2}\geq k_1(\frac{2C_{\mu}}{\kappa})^{1/\delta},\ee
 the operator $\mathrm{Id} - (\mathscr{L}^{2}_{[a]\times[b]})^{-1}\mathscr{K}_{[a]\times[b]}$ is invertible and from \eqref{amel3} we get 
\begin{eqnarray*} \| B_{[a]}^{[b]} \|  &=& \|\left( \mathrm{Id} - (\mathscr{L}^{2}_{[a]\times[b]})^{-1}\mathscr{K}_{[a]\times[b]}\right)^{-1} \| \| \left(  \mathscr{L}^{2}_{[a]\times[b]} \right)^{-1}\left(A_{[a]}^{[b]}\right)\| \\
& \leq & 2  \| \left(  \mathscr{L}^{2}_{[a]\times[b]} \right)^{-1}\left(A_{[a]}^{[b]}\right)\|\,,
\end{eqnarray*}
Hence in this case 
\be \label{delort2} |B|_{s,\b+} \leq \frac{2}{\kappa}|A|_{s,\b}\ee

\noindent {\bf Step 3} : suppose  $A \in \msb \cap \Upsilon_{s,\beta}^{3}(k_{1},k_{2})$. The only nonzero blocks $A_{[a]}^{[b]}$ correspond to weights $w_{a}$ and $w_{b}$ such that 
$$ \max(w_{a},w_{b}) \leq k_{1} \min(w_{a},w_{b}) \; \mbox{and} \; \max(w_{a},w_{b})\leq k_{2} \,,$$
hence there are only finitely many such blocks. In this case, for any $j \in[a]$ and $l \in [b]$ we have
\be |B_{j}^{l}| = \left| \frac i{\langle k,\om(\r)\rangle \ +\eps\mu_j-\mu_l}\right| |{A}_j^l | \leq \frac{1}{\kappa(1+|w_{a}-w_{b}|) } |{A}_j^l |\ee
A majoration of the coefficients gives a poor majoration of the operator norm of a matrix, but it is sufficient here since the number of nonzero blocks (and their size, see \eqref{block}) is finite~:
 \be \| B_{[a]}^{[b]} \| \leq( \frac{(C_{b}\max(w_{a},w_{b}))^{d^{*}/2}}{\kappa(1+|w_{a}-w_{b}|) }\|A_{[a]}^{[b]} \|\,,
 \ee
 hence  $B \in \msb^{+}$ and
 \be \label{delort3} |B|_{s,\b+} \leq \frac{(C_{b}k_{2})^{d^{*}/2}}{\kappa} |A|_{s,\b}\,.\ee
Collecting \eqref{delort1}, \eqref{delort2} and \eqref{delort3} and taking into account \eqref{k1}, \eqref{k2} leads to the  result.
\endproof

\end{document}